\documentclass[12pt]{article}
\usepackage{amssymb,bm}
\usepackage{amsmath}
\usepackage{amsthm}
\usepackage{graphicx}
\usepackage{hyperref}
\hypersetup{pdfborder=0 0 0}

\newtheorem{theorem}{{\sc Theorem}}[section]

\newtheorem{definition}[theorem]{Definition}

\newcommand{\bb}[1]{\mathbb{ #1}}

\bmdefine\Bone{1}

\newcommand{\bra}[1]{\overline{#1}}

\newcommand{\cof}{\mathrm{cof}}
\newcommand{\rank}{\mathrm{rank}}

\newcommand{\nth}[1]{\displaystyle\frac{1}{#1}}

\newcommand{\Md}{\partial}
\renewcommand{\Hat}[1]{\widehat{#1}}
\newcommand{\Tld}[1]{\widetilde{#1}}

\def\XXint#1#2#3{{\setbox0=\hbox{$#1{#2#3}{\int}$ }
\vcenter{\hbox{$#2#3$ }}\kern-.6\wd0}}

\newcommand{\im}{\mathfrak{Im}}
\newcommand{\re}{\Re\mathfrak{e}}

\newcommand{\rhs}{right-hand side}

\newcommand{\nbh}{neighborhood}
\newcommand{\IFF}{if and only if }

\newcommand{\Ga}{\alpha}
\newcommand{\Gb}{\beta}
\newcommand{\Gd}{\delta}
\newcommand{\Ge}{\epsilon}

\newcommand{\Gve}{\varepsilon}

\newcommand{\Gg}{\gamma}

\newcommand{\Gl}{\lambda}

\newcommand{\Gth}{\theta}

\newcommand{\Gs}{\sigma}

\newcommand{\Go}{\omega}

\newcommand{\Gz}{\zeta}
\newcommand{\GD}{\Delta}

\newcommand{\GS}{\Sigma}

\bmdefine\BGa{\alpha}
\bmdefine\BGb{\beta}
\bmdefine\BGd{\delta}
\bmdefine\BGe{\epsilon}
\bmdefine\BGve{\varepsilon}
\bmdefine\BGf{\phi}
\bmdefine\BGvf{\varphi}
\bmdefine\BGg{\gamma}
\bmdefine\BGc{\chi}
\bmdefine\BGi{\iota}
\bmdefine\BGk{\kappa}
\bmdefine\BGl{\lambda}
\bmdefine\BGn{\eta}
\bmdefine\BGm{\mu}
\bmdefine\BGv{\nu}
\bmdefine\BGp{\pi}
\bmdefine\BGth{\theta}
\bmdefine\BGvth{\vartheta}
\bmdefine\BGr{\rho}
\bmdefine\BGvr{\varrho}
\bmdefine\BGs{\sigma}
\bmdefine\BGvs{\varsigma}
\bmdefine\BGt{\tau}
\bmdefine\BGj{\tau}
\bmdefine\BGu{\upsilon}
\bmdefine\BGo{\omega}
\bmdefine\BGx{\xi}
\bmdefine\BGy{\psi}
\bmdefine\BGz{\zeta}
\bmdefine\BGD{\Delta}
\bmdefine\BGF{\Phi}
\bmdefine\BGG{\Gamma}
\bmdefine\BGL{\Lambda}
\bmdefine\BGP{\Pi}
\bmdefine\BGT{\Theta}
\bmdefine\BGS{\Sigma}
\bmdefine\BGU{\Upsilon}
\bmdefine\BGO{\Omega}
\bmdefine\BGX{\Xi}
\bmdefine\BGY{\Psi}

\bmdefine\BFM{\mathfrak{M}}
\bmdefine\BFb{\mathfrak{b}}
\bmdefine\BFk{\mathfrak{k}}
\bmdefine\BFm{\mathfrak{m}}
\bmdefine\BFu{\mathfrak{u}}
\bmdefine\BFv{\mathfrak{v}}

\newcommand{\CA}{{\mathcal A}}

\newcommand{\CP}{{\mathcal P}}

\newcommand{\CR}{{\mathcal R}}

\newcommand{\CW}{{\mathcal W}}

\bmdefine\BCA{{\mathcal A}}
\bmdefine\BCB{{\mathcal B}}
\bmdefine\BCC{{\mathcal C}}
\bmdefine\BCD{{\mathcal D}}
\bmdefine\BCE{{\mathcal E}}
\bmdefine\BCF{{\mathcal F}}
\bmdefine\BCG{{\mathcal G}}
\bmdefine\BCH{{\mathcal H}}
\bmdefine\BCI{{\mathcal I}}
\bmdefine\BCJ{{\mathcal J}}
\bmdefine\BCK{{\mathcal K}}
\bmdefine\BCL{{\mathcal L}}
\bmdefine\BCM{{\mathcal M}}
\bmdefine\BCN{{\mathcal N}}
\bmdefine\BCO{{\mathcal O}}
\bmdefine\BCP{{\mathcal P}}
\bmdefine\BCQ{{\mathcal Q}}
\bmdefine\BCR{{\mathcal R}}
\bmdefine\BCS{{\mathcal S}}
\bmdefine\BCT{{\mathcal T}}
\bmdefine\BCU{{\mathcal U}}
\bmdefine\BCV{{\mathcal V}}
\bmdefine\BCW{{\mathcal W}}
\bmdefine\BCX{{\mathcal X}}
\bmdefine\BCY{{\mathcal Y}}
\bmdefine\BCZ{{\mathcal Z}}

\bmdefine\Bzr{ 0}
\bmdefine\Ba{ a}
\bmdefine\Bb{ b}
\bmdefine\Bc{ c}
\bmdefine\Bd{ d}
\bmdefine\Be{ e}
\bmdefine\Bf{ f}
\bmdefine\Bg{ g}
\bmdefine\Bh{ h}
\bmdefine\Bi{ i}
\bmdefine\Bj{ j}
\bmdefine\Bk{ k}
\bmdefine\Bl{ l}
\bmdefine\Bm{ m}
\bmdefine\Bn{ n}
\bmdefine\Bo{ o}
\bmdefine\Bp{ p}
\bmdefine\Bq{ q}
\bmdefine\Br{ r}
\bmdefine\Bs{ s}
\bmdefine\Bt{ t}
\bmdefine\Bu{ u}
\bmdefine\Bv{ v}
\bmdefine\Bw{ w}
\bmdefine\Bx{ x}
\bmdefine\By{ y}
\bmdefine\Bz{ z}
\bmdefine\BA{ A}
\bmdefine\BB{ B}
\bmdefine\BC{ C}
\bmdefine\BD{ D}
\bmdefine\BE{ E}
\bmdefine\BF{ F}
\bmdefine\BG{ G}
\bmdefine\BH{ H}
\bmdefine\BI{ I}
\bmdefine\BJ{ J}
\bmdefine\BK{ K}
\bmdefine\BL{ L}
\bmdefine\BM{ M}
\bmdefine\BN{ N}
\bmdefine\BO{ O}
\bmdefine\BP{ P}
\bmdefine\BQ{ Q}
\bmdefine\BR{ R}
\bmdefine\BS{ S}
\bmdefine\BT{ T}
\bmdefine\BU{ U}
\bmdefine\BV{ V}
\bmdefine\BW{ W}
\bmdefine\BX{ X}
\bmdefine\BY{ Y}
\bmdefine\BZ{ Z}

\title{Reconstructing Stieltjes functions from their approximate values: a
  search for a needle in a haystack}
\author{Yury Grabovsky}
\begin{document}
\maketitle
\begin{abstract}
  Material response of real, passive, linear, time-invariant media to external
  influences is described by complex analytic functions of frequency that can
  always be written in terms of Stieltjes functions---a special class of
  analytic functions mapping complex upper half-plane into
  itself. Reconstructing such functions from their experimentally measured
  values at specific frequencies is one of the central problems that we
  address in this paper. A definitive reconstruction algorithm that produces a
  certificate of optimality as well as a graphical representation of the
  uncertainty of reconstruction is proposed. Its effectiveness is demonstrated
  in the context of the electrochemical impedance spectroscopy.
\end{abstract}
\tableofcontents
\section{Introduction}
\setcounter{equation}{0} 
\label{sec:intro} 
Three fundamental physical principles: linearity, time-invariance, and
passivity are responsible for the ubiquity of Stieltjes functions in physics
and engineering. \emph{Stieltjes class} refers to a special class of complex
analytic functions that describe the response of linear media or devices to
external influences. If $E(t)$ denotes such an influence, and $J(t)$ the
response, then the linear, time-invariant dependence of $J(t)$ on
$E(t)$ could be formally written (without regard to the function spaces to which $E(t)$ and
$J(t)$ may belong) as
\begin{equation}
  \label{LTI}
  J(t)=\Gg_{0}E(t)+\int_{-\infty}^{t}a(t-\tau)E(\tau)d\tau,
\end{equation}
where the causality principle, limiting the dependence of $J(t)$ only on the
present and past values of $E(\tau)$, has been applied. For a mathematically rigorous
discussion of convolution-type formulas, like (\ref{LTI}) we refer the reader
to many treatises on linear systems theory, e.g., \cite{wobe65,zema72}.

Due to the resemblance of the integral
in (\ref{LTI}) to a convolution, it is convenient to extend the memory kernel $a(s)$ to
negative values of $s$ by zero
\[
a_{0}(s)=
\begin{cases}
  a(s),&s\ge 0,\\
  0,&s<0,
\end{cases}
\]
and rewrite (\ref{LTI}) as a convolution
\begin{equation}
  \label{LTIconv}
  J(t)=\Gg_{0}E(t)+\int_{-\infty}^{\infty}a_{0}(t-\tau)E(\tau)d\tau.
\end{equation}
Assuming now that $a_{0}\in L^{1}(\bb{R})$ and $\{E,J\}\subset L^{2}(\bb{R})$
we can take the Fourier transform of (\ref{LTIconv}): 
\begin{equation}
  \label{FLTI}
  \Hat{J}(\Go)=(\Gg_{0}+\Hat{a}_{0}(\Go))\Hat{E}(\Go).
\end{equation}
Two different definitions of the Fourier transform are common in physics,
depending on the representation of the input $E(t)$ as a superposition of
``elementary harmonics''. In signal processing and electrical circuit theory
the elementary harmonics are functions $e^{i\Go t}$, leading to the
representation
\[
E(t)=\nth{2\pi}\int_{-\infty}^{\infty}\Hat{E}(\Go)e^{i\Go t}d\Go,\qquad
\Hat{E}(\Go)=\int_{-\infty}^{\infty}E(t)e^{-i\Go t}dt.
\]
In electromagnetics the elementary harmonics are the plane waves $e^{i(\Bk\cdot\Bx-\Go t)}$. 
In this case one uses
\[
E(t)=\nth{2\pi}\int_{-\infty}^{\infty}\Hat{E}(\Go)e^{-i\Go t}d\Go,\qquad
\Hat{E}(\Go)=\int_{-\infty}^{\infty}E(t)e^{i\Go t}dt.
\]
In the former case (e.g. impedance of electrical circuits) causality,
$a_{0}(s)=0$, when $s<0$, implies that $b(\Go)=\Gg_{0}+\Hat{a}_{0}(\Go)$ is
analytic in the lower half-plane of the complex $\Go$-plane, in the latter
(e.g. complex dielectric permittivity), $b(\Go)$ is analytic in the upper
half-plane. In each case the fact that the memory kernel $a(s)$ is a
real-valued function implies that $b(\Go)$ has the symmetry
\begin{equation}
  \label{R2sym}
  \bra{b(\Go)}=b(-\bra{\Go}).
\end{equation}
The passivity principle, that says that the medium can only absorb or
dissipate energy is a much more delicate condition leading to the
nonnegativity of the real or imaginary parts of functions related to
$b(\Go)$. In one way or another in each and every application the description
of the linear, time-invariant, passive media response can be formulated in
terms of functions from the Stieltjes class\footnote{There is no universal
  agreement on the names attached to various related classes of analytic
  functions. That is why we give a full formal definition here.} $\mathfrak{S}$.
\begin{definition}
  \label{def:Stcl}
We say that a complex function $f$ analytic in $\bb{C}\setminus\bb{R}_{+}$
belongs to the Stieltjes class $\mathfrak{S}$ if it
is either a nonnegative real constant or has the following three properties.
\begin{enumerate}
\item[(i)] $\im(f(z))>0$ for all $z\in\bb{C}$ with $\im(z)>0$;
\item[(ii)] $f(x)>0$ for all $x<0$;
\item[(iii)] $\bra{f(z)}=f(\bra{z})$.
\end{enumerate}
\end{definition}
For example, the complex electromagnetic permittivity $\Gve(\Go)$ of
dielectrics can be written as $\Gve(\Go)=f(\Go^{2})$, where $f\in\mathfrak{S}$
and $\im(\Go)>0$ \cite{lali60:8,feyn64}. Both the complex impedance and
admittance functions $Z(\Go)$ and $Y(\Go)$, respectively, of electrical circuits
made of resistors, capacitors and inductive coils can be written as
$Z(\Go)=i\Go f(\Go^{2})$, where $f\in\mathfrak{S}$ and $\im(\Go)<0$
\cite{brune31}. In high energy physics it is the energy (or momentum) that
plays the role of the complex variable and the scattering amplitude is the Stieltjes function
\cite{khur57,macd59,hmsw61,nuss72,capr74}. In the theory of binary conducting composites the
dependence of the effective conductivity $\Gs^{*}$ of the composite on the
ratio $h=\Gs_{1}/\Gs_{2}$ of the conductivities of two constituents is also
expressible in terms of Stieltjes functions, \cite{berg78,milt81b,gopa83,lipt01}
$\Gs^{*}/\Gs_{1}=1+(1-h)f(-h)$, where $f\in\mathfrak{S}$.
There are many other applications (see e.g., \cite{milt17}), where the models
are linear, and causality, time-invariance, and passivity (together with real
values of the memory kernel) lead to system descriptions in terms of functions
from the Stieltjes class $\mathfrak{S}$.

In this paper we consider the central \emph{discrete} problem of the theory of
Stieltjes functions that arises in all applications: the identification of
$f\in\mathfrak{S}$ from $n$ measurements at the $n$ distinct points
$\{z_{1},\ldots,z_{n}\}\subset\bb{H}_{+}$, where $\bb{H}_{+}$ denotes the
complex upper half-plane. The analyticity of $f\in\mathfrak{S}$ places
constraints on the values $f(z_{j})$. It turns out that the constraints are so
delicate that even if one truncates the infinite decimal representations of
the values $w_{j}=f(z_{j})$ in order to store them as floating point numbers
in a computer, one violates these constraints when $n\ge 15$. In most
applications the values $w_{j}=f(z_{j})$ are obtained through experimental
measurements where the noise level is much larger than the round-off errors in
floating point arithmetic. In view of these considerations the central problem
is not the recovery of $f\in\mathfrak{S}$ from its exact values
$w_{j}=f(z_{j})$, but rather the minimization of the sum of squares
\begin{equation}
  \label{lsqStielt}
  \GS(\Bw,\Bz)=\inf_{f\in\mathfrak{S}}\sum_{j=1}^{n}|f(z_{j})-w_{j}|^{2}
\end{equation}
for a given set of noisy measurements $\Bw\in\bb{C}^{n}$. The problem of
solving (\ref{lsqStielt}) bears only superficial resemblance to the classical
linear least squares problem. The main difficulty is that the Stieltjes class
$\mathfrak{S}$ is not a vector space, but a convex cone.

In various guises this problem has been studied continuously for almost a
century see e.g.,
\cite{hmsw61,ciulli69,aogr95,mem97,digr01,ouch06,zhch09,ocg12,ou14,bouk20,sriv20}.
Yet, so far, no definitive algorithm for solving (\ref{lsqStielt}) has
emerged, and new algorithms and new papers on the subject continue to appear
with unerring regularity (e.g., \cite{bouk95,lspv05,mash16,zpsp20}, to give a
taste). In this paper we propose such a definitive algorithm, described in
Section~\ref{sec:lsq}, that is aimed to settle the question once and for
all. The algorithm comes with a ``certificate of optimality'' based on the
work of I.~Caprini \cite{capr74,capr79,capr80,capr81}. The FORTRAN
implementation of the algorithm is available from Github \cite{Stgh21}.  The
method is easily extendable to weighted sums of squares as in Caprini's
papers.

The main issue lies in intricacies of the geometry of the \emph{interpolation body}
\begin{equation}
  \label{Vofzdef}
  V(\Bz)=\{(f(z_{1}),\ldots,f(z_{n}))\in\bb{C}^{n}: f\in\mathfrak{S}\},\quad
\Bz=(z_{1},\ldots,z_{n}),
\end{equation}
which is known to be a closed convex cone in $\bb{C}^{n}$ with non-empty
interior. In practice, however, $V(\Bz)$ is massively dimensionally
degenerate, shaped very much like a needle or a sword. Even for modest values
of $n$ the smallest thickness of $V(\Bz)$ is well below double precision
floating point arithmetic. The proposed algorithm harnesses this dimensional
degeneracy and turns it from a curse into a blessing. The algorithm produces
not only the solution $f\in\mathfrak{S}$ of (\ref{lsqStielt}), but also shows
the uncertainty associated with the given data (see
Figure~\ref{fig:Voigt}). Typical for analytic continuation problems the
uncertainty balloons and explodes once one goes outside of the frequency range
containing the measurements \cite{deto18,trefe19,bdmh19,grho-annulus,grho-gen}
(see Figure~\ref{fig:Voigt_minus}).

The algorithm described in Section~\ref{sec:lsq} is an outcome of the
understanding of the geometry of the interpolation body $V(\Bz)$ discussed in
Sections~\ref{sec:prelim} and \ref{sec:needle} as well as the optimality
conditions described in Theorem~\ref{th:Caprini}. The key ingredient in the
algorithm is the use of the local minima of the Caprini function to augment
the ad-hoc basis of the space of Stieltjes functions. The final step is based
on the realization that the near-optimal solution for a given noisy data is an
optimal solution for ``nearby data'' representing a slightly different
realization of the noise. The FORTRAN implementation of the algorithm
is publicly available \cite{Stgh21}.

The fact that points $z_{j}$ lie in the upper half-plane, and not on the real
line is essential for our analysis. When some or all of the points $z_{j}$ lie
on the negative semi-axis a modification of our analysis given in
\cite[Ch.~V.3]{Krein:1977:MMP} and \cite{krnu98} is necessary. Complementary
to the setting of this paper is the situation where the imaginary part of
$f(z)$ is known on a finite sub-interval of the positive real axis, while the
real part is known only at finitely many points in that same interval. Another
complementary situation is when measurements are done in the time domain. The
former is studied in \cite{mem97}, the latter is addressed in
\cite[Chapter~6]{milt17} and \cite{mmp21}, where the collapse onto a needle is
reflected in the fact that the time dependent bounds for an appropriate input
and at a particular time almost coincide: one is viewing from a direction
along the line of the needle \cite{gmpc20}.

This paper is structured as follows. We begin our discussion with the
recollection of known results about Stieltjes functions in
Section~\ref{sec:prelim}. In Section~\ref{sec:needle} we show that the
interpolation body $V(\Bz)$ is shaped like a needle or maybe like a sword. (Our
language has an inadequate vocabulary limited to two and three-dimensional
shapes.) In Sections~\ref{sec:lsq} and \ref{sec:SM} we describe the algorithm.
The performance of the algorithm is demonstrated in Section~\ref{sec:Voigt} in
the context of electrochemistry, where the processes of corrosion and
electrolysis that occur in batteries and in many other natural and man-made
systems can be modeled by Voigt circuits---electrical circuits made only of
resistors and capacitors \cite{ssk93,bafa00,bamc05}. The electrochemical
impedance spectrum (EIS) function $Z(\Go)$ can then be written as $f(-i\Go)$
for some $f\in\mathfrak{S}$. Thus, the values $Z_{j}=f(-i\Go_{j})$,
$j=1,\ldots,n$ can be measured experimentally at particular frequencies
$\Go_{1},\ldots,\Go_{n}$. Our algorithm takes noisy measurements of
$Z_{1},\ldots,Z_{n}$ as the input and generates physically admissible EIS
function $Z(\Go)$, representing it both numerically and as the explicit
complex impedance function of a small Voigt circuit. It also displays the
certificate of optimality as well as the uncertainty of reconstruction of the
EIS function for the specific data. Figure~\ref{fig:Voigt} shows the typical
graphical output of the algorithm.

\section{Preliminaries and background}
\setcounter{equation}{0} 
\label{sec:prelim} 
\subsection{The Nevanlinna-Pick theorem for Stieltjes functions}
We recall two equivalent characterizations of the Stieltjes class
$\mathfrak{S}$. One exhibits the centrality of
property (i) in Definition~\ref{def:Stcl}, which is an expression of passivity
in frequency domain. The other gives an explicit representation of all
Stieltjes functions.  Let $ \bb{H}_{+}=\{z\in\bb{C}:\im(z)>0\} $ denote the
complex upper half-plane.
\begin{definition}
  \label{def:Ncl}
We say that $f(z)$ analytic in $\bb{H}_{+}$ is a Nevanlinna function if it
is either a real constant function or
 $\im(f(z))>0$ for all $z\in\bb{H}_{+}$.
\end{definition}
Other names for this class, such as Herglotz functions, Pick functions, and
R-functions are also used by various communities.
\begin{theorem}
\label{th:SviaR}
  $f\in\mathfrak{S}$ \IFF both $f$ and $z\mapsto zf(z)$ are Nevanlinna functions.
\end{theorem}

As a corollary we see that the Stieltjes class has an involutive symmetry 
\begin{equation}
  \label{invsym}
  f(z)\mapsto-\nth{zf(z)}.
\end{equation}
The second characterization of $\mathfrak{S}$ is more explicit.
\begin{theorem}[Stieltjes]
\label{th:Stieltjes}
  $f\in\mathfrak{S}$ \IFF there exists $\Gg\ge 0$ and a positive Radon measure $\Gs$
  on $[0,+\infty)$, such that
  \begin{equation}
    \label{Stielrep}
    f(z)=\Gg+\int_{0}^{\infty}\frac{d\Gs(t)}{t-z},\qquad
\int_{0}^{\infty}\frac{d\Gs(t)}{1+t}<+\infty.
  \end{equation}
\end{theorem}
The proof of both theorems can be found in \cite[Chapter~III, Addendum]{akhi21} 
% Problems 3 and 4, p. 159-161
or in \cite[Addendum, Section~2]{Krein:1977:MMP}. %p. 522-525
We remark that given $f\in\mathfrak{S}$ we
have
\begin{equation}
  \label{recovery}
  \Gg=\lim_{z\to\infty}f(z),\qquad \Gs(x)=\nth{\pi}\lim_{y\to 0^{+}}\im f(x+iy),
\end{equation}
where the second limit above is understood in the sense of distributions.

Our goal is the recovery of a Stieltjes function $f$ from its approximately
known values 
$f(z_{1}),\ldots,f(z_{n})$ at distinct points
$\{z_{1},\ldots,z_{n}\}\subset\bb{H}_{+}$. In this regard
we recall a well-known Nevanlinna-Pick theorem that, combined with
Theorem~\ref{th:SviaR}, gives a criterion for
$\Bw\in\bb{C}^{n}$ to lie in the interpolation body $V(\Bz)$, given by (\ref{Vofzdef}).
\begin{theorem}[Nevanlinna-Pick]
\label{th:NP}
Let $\{z_{1},\ldots,z_{n}\}\subset\bb{H}_{+}$ be all distinct and
$\Bw\in\bb{C}^{n}$. Then $\Bw\in V(\Bz)$ \IFF the Nevanlinna-Pick matrices
$\BN(\Bz,\Bw)$ and $\BP(\Bz,\Bw)$ are nonnegative definite, where
  \begin{equation}
    \label{NP}
    N_{jk}(\Bz,\Bw)=\frac{w_{j}-\bra{w_{k}}}{z_{j}-\bra{z_{k}}},\qquad
P_{jk}(\Bz,\Bw)=\frac{z_{j}w_{j}-\bra{z_{k}w_{k}}}{z_{j}-\bra{z_{k}}}.
  \end{equation}
  Moreover, if $\Bw\in\Md V(\Bz)$, so that either $\rank(\BN(\Bz,\Bw))<n$ or
  $\rank(\BP(\Bz,\Bw))<n$, then there is a unique rational function
  $f\in\mathfrak{S}$, such that $w_{j}=f(z_{j})$, $j=1,\ldots,n$.
\end{theorem}
For the proof see e.g., \cite[Ch.~16--18]{simon19} (see also \cite{krnu98}).

\subsection{Bounds on Stieltjes function values}
\label{sub:Sbds}
The question we want to address now is about the freedom one has for the value
$w=f(z)$, provided $f\in\mathfrak{S}$ and satisfies $f(z_{j})=w_{j}$,
$j=1,\ldots,n$. This freedom is represented by the admissible
set of values
\begin{equation}
  \label{Admreg}
  \CA(z;\Bz,\Bw)=\{f(z):f\in\mathfrak{S},\ f(z_{j})=w_{j},\ j=1,\ldots,n\}.
\end{equation}
Such admissible sets are well-understood and widely used in the context of
effective properties of composite materials
\cite{gopa83,gopa85,gold95a,mmdgbg96,chgo98}. Our analysis is inspired by the
one in \cite{mi81} and reaches somewhat similar conclusions. However, it is
based on Theorem~\ref{th:NP} rather than the explicit representation of
Stieltjes functions from Theorem~\ref{th:Stieltjes}, used in prior work.
The question of bounds on values of Stieltjes functions in the case when the
spectral measure $\Gs$ is known in an interval of frequencies is addressed in
\cite{mem97}. The bounds in the case when the phase of the analytic function
is known on a part of the boundary, and on the modulus on the remaining part
have been derived in \cite{acd20} by means of a modified Nevanlinna-Pick problem.

Let us assume that the data $\Bw$ lies in the interior of $V(\Bz)$. By
Theorem~\ref{th:NP} the matrices $\BN(\Bz,\Bw)$ and $\BP(\Bz,\Bw)$, given by
(\ref{NP}), are positive definite. Then, by Sylvester's criterion
(e.g., \cite{hojo85}) we obtain
that the $\BN([\Bz,z],[\Bw,w])$ and $\BP([\Bz,z],[\Bw,w])$ matrices
corresponding to the extended data $([\Bz,z],[\Bw,w])$ are positive definite
\IFF
\begin{equation}
  \label{detineq}
  \det\BN([\Bz,z],[\Bw,w])>0,\qquad\det\BP([\Bz,z],[\Bw,w])>0.
\end{equation}
We can make inequalities (\ref{detineq}) explicit, since the determinants above are
quadratic functions of $w$. 
%Using the Schur complement determinant formula 
Expanding the determinants with respect to the last
column and the last row, so that $w$ enters explicitly, 
we obtain
\[
\det\BN([\Bz,z],[\Bw,w])=\frac{\im(w)}{\im(z)}\det\BN(\Bz,\Bw)-
\Ga|w|^{2}+2\re(aw)-\Gb,
\]
where
\[
\Ga=(\cof(\BN)\BGx(z),\BGx(z)),\qquad a=(\cof(\BN)\BGx(z),\BGn(z)),\qquad\Gb=(\cof(\BN)\BGn(z),\BGn(z)),
\]
and $\BN$ stands for $\BN(\Bz,\Bw)$, and
\[
\xi_{k}(z)=\frac{1}{z-\bra{z_{k}}},\qquad\eta_{k}(z)=\frac{\bra{w_{k}}}{z-\bra{z_{k}}},
\qquad k=1,\ldots,n.
\]
We conclude that $\det\BN([\Bz,z],[\Bw,w])>0$ \IFF $|w-w^{(N)}(z)|< r_{N}(z)$,
where
\begin{equation}
  \label{wrN}
  w^{(N)}(z)=\frac{\bra{a}}{\Ga}+i\frac{\det\BN(\Bz,\Bw)}{2\Ga\im(z)},\qquad
r_{N}(z)^{2}=|w^{(N)}(z)|^{2}-\frac{\Gb}{\Ga}.
\end{equation}
A similar analysis for the $\BP$-matrix gives $|w-w^{(P)}(z)|<r_{P}(z)$, where
\begin{equation}
  \label{wrP}
  w^{(P)}(z)=\frac{\bra{a'}}{\Ga'}+i\frac{\bra{z}\det\BP(\Bz,\Bw)}{2\Ga'\im(z)},\qquad
r_{P}(z)^{2}=|w^{(P)}(z)|^{2}-\frac{\Gb'}{\Ga'},
\end{equation}
and
\[
\Ga'=(\cof(\BP)\BGx'(z),\BGx'(z)),\quad a'=(\cof(\BP)\BGx'(z),\BGn'(z)),\quad
\Gb'=(\cof(\BP)\BGn'(z),\BGn'(z)),
\]
\[
\BGx'(z)=z\BGx(z),\qquad\BGn'(z)=z\BGn(z)-\bra{\Bw},\qquad\BP=\BP(\Bz,\Bw)
\]

Let us now estimate $r_{N}(z)$. (The estimate for $r_{P}(z)$ would be fully
analogous.) The key observation is the inequality between $\Ga$, $\Gb$ and $a$:
$|a|^{2}\le\Ga\Gb$. Then
\[
r_{N}(z)^{2}=\frac{|a|^{2}-\Ga\Gb}{\Ga^{2}}+\rho^{2}
-\frac{2\im(a)\rho}{\Ga}\le 2\rho\left(\rho-\frac{\im(a)}{\Ga}\right)=
2\rho\im(w^{(N)}(z)),
\]
where
\[
  \rho=\frac{\det\BN(\Bz,\Bw)}{2\Ga\im(z)}.
\]
Thus, we have obtained the estimate
\begin{equation}
  \label{Nrest}
  r_{N}(z)^{2}\le\frac{\im(w^{(N)}(z))}{\im(z)}(\BN(\Bz,\Bw)^{-T}\BGx(z),\BGx(z))^{-1}.
\end{equation}
A similar calculation for the $\BP$ matrix gives the estimate
\begin{equation}
  \label{PRest}
  r_{P}(z)^{2}\le\frac{\im(zw^{(P)}(z))}{\im(z)}(\BP(\Bz,\Bw)^{-T}\BGx'(z),\BGx'(z))^{-1}.
\end{equation}
The main feature of matrices $\BN(\Bz,\Bw)$ and $\BP(\Bz,\Bw)$ is the exponential decay of their
eigenvalues due to their rank two displacement structure \cite{beto17}:
\begin{equation}
  \label{ND}
  \BD(\Bz)\BN(\Bz,\Bw)-\BN(\Bz,\Bw)\BD(\Bz)^{*}=\Bw\otimes\Bone-\Bone\otimes\bra{\Bw},
\end{equation}
\begin{equation}
  \label{PD}
\BD(\Bz)\BP(\Bz,\Bw)-\BP(\Bz,\Bw)\BD(\Bz)^{*}=\BD(\Bz)\Bw\otimes\Bone-\Bone\otimes\BD(\bra{\Bz})\bra{\Bw},
\end{equation}
where $\BD(\Bz)$ is a diagonal matrix with numbers $z_{j}$ on the main
diagonal and $\Bone$ is a vector of ones.

If the vector $\BGx(z)$ has a substantial projection onto the space spanned by
the eigenvectors of $\BN(\Bz,\Bw)$ and $\BP(\Bz,\Bw)$ with exponentially small
eigenvalues, then $(\BN(\Bz,\Bw)^{-T}\BGx(z),\BGx(z))$ and
$(\BP(\Bz,\Bw)^{-T}\BGx'(z),\BGx'(z))$ will be exponentially large (as functions
of $n$). This shows that $r_{N}(z)$ and $r_{P}(z)$ can easily become
exponentially small even for relatively small values of $n$. In fact,
$r_{N}(z)=0$ or $r_{P}(z)=0$ (or both) whenever $\Bw\in\Md V(\Bz)$. This may
lead one to think that fixing more than 15--20 values of a Stieltjes function
determines it for all practical intents and purposes. The truth is more
nuanced. It depends very strongly on the relative location of $z$ and $z_{j}$
and on the exact location of $\Bw\in V(\Bz)$ relative to $\Md V(\Bz)$. 
Formally, $V(\Bz)$ is a closed convex cone in $\bb{C}^{n}$ with non-empty
interior. In practice, its geometry resembles that of a thin knife blade,
rather than a party hat, so that very small random perturbations of points in
the interior of $V(\Bz)$ will throw them outside. In other words, no matter
where the point $\Bw$ is in $V(\Bz)$, it is never far from $\Md V(\Bz)$,
where, as we have just observed, the region of admissible values
$\CA(z;\Bz,\Bw)$ degenerates to a point. What is somewhat counter-intuitive is
that for points $\Bw$ in the interior of $V(\Bz)$ the set $\CA(z;\Bz,\Bw)$ can
be rather large, depending on the location of $z$ relative to points $z_{j}$.
\begin{figure}[t]
  \centering
  \includegraphics[scale=0.3]{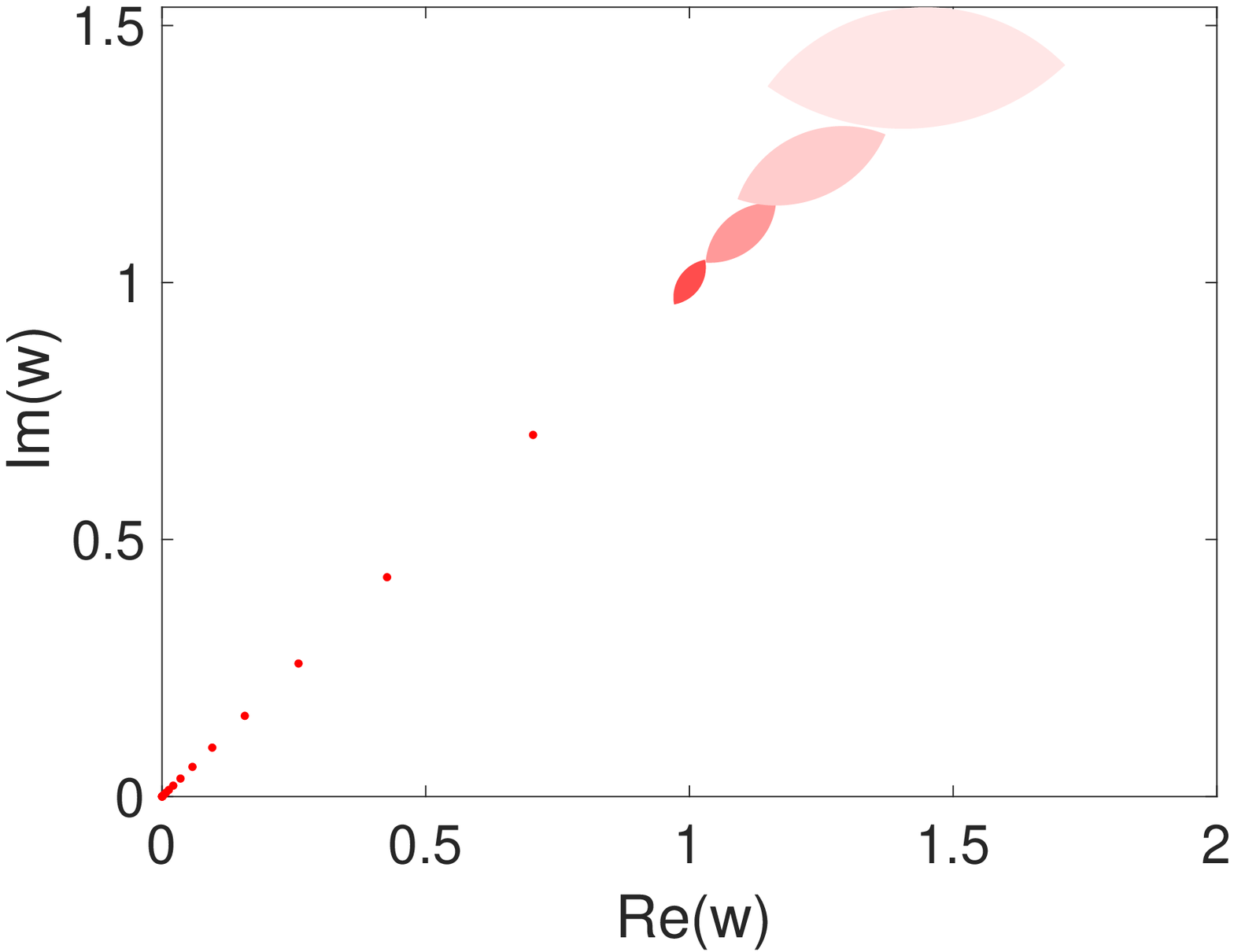}~~~~~~~~~
\includegraphics[scale=0.3]{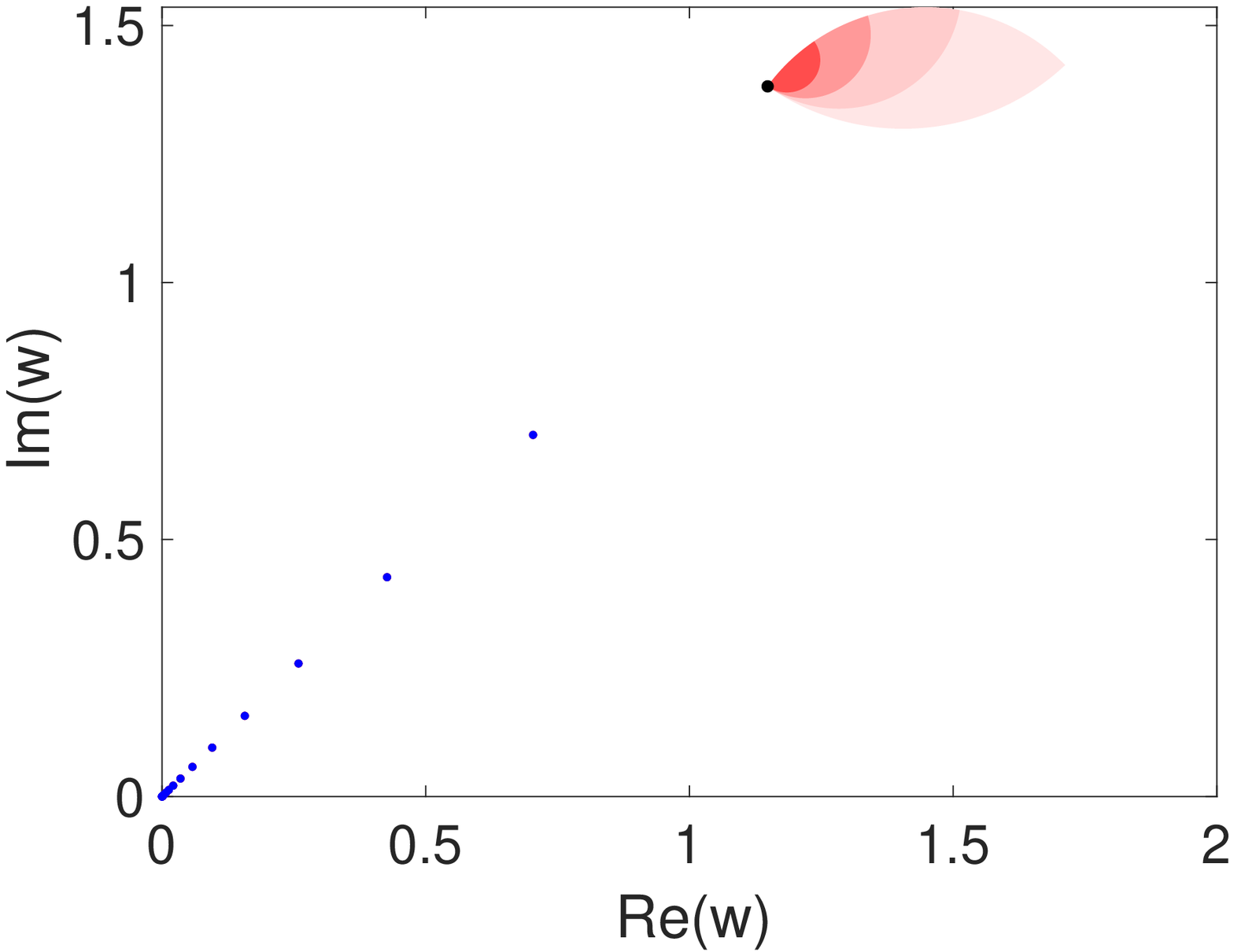}
% Made by rsch/kramers-kronig/Stieltjes/MPStieltjes/MPUQ0.m
  \caption{Dependence of the admissible set $\CA(z;\Bz,\Bw)$ on the location
    of $z$, relative to the data $z_j$ (left) and on the location of $\Bw$,
    relative to $\Md V(\Bz)$ (right).}
  \label{fig:lenses}
\end{figure}
The left panel of Figure~\ref{fig:lenses} illustrates this effect in the simple example
\begin{equation}
  \label{sqrtexample0}
  z_{j}=ie^{0.01+j},\quad w_{j}=f(z_{j}),\ j=0,1,\ldots,19,\quad f(z)=\nth{\sqrt{-z}}.
\end{equation}
We see how the shaded lens-shaped regions grow in size as the point $z$,
taking values $i/2$, $i/2.4$, $i/3$, and $i/4$ moves ``away'' from the data
$z_{j}$, given in (\ref{sqrtexample0}). Our discussion also shows that if we
move $\Bw$ from the interior of $V(\Bz)$ to its boundary the admissible set
will shrink to a point. The right panel of Figure~\ref{fig:lenses} illustrates
this effect when we move from $\Bw$, given in (\ref{sqrtexample0}), which lies
in the interior of $V(\Bz)$, to $\Md V(\Bz)$ along any random direction $\Bu$,
which we have chosen (arbitrarily) to have all components equal to $-1$. The
corresponding point $\Tld{\Bw}\in\Md V(\Bz)$ satisfies $
|\Bw-\Tld{\Bw}|/|\Bw|<10^{-4}, $ as we have verified numerically. In the right
panel of Figure~\ref{fig:lenses} we plotted the original points $w_{j}$ in red
and the perturbed points $\Tld{w}_{j}$ in blue, except one cannot see a
difference between them in the figure. The set $\CA(z;\Bz,\Tld{\Bw})$
degenerates to a point shown in black, while the the sets
$\CA_{t}=\CA(z;\Bz,t\Tld{\Bw}+(1-t)\Bw)$ for three intermediate values of $t$
are shown by progressively darker shading. The values we have chosen are
$t_{1}=1-2\cdot10^{-5}$, $t_{2}=1-7\cdot10^{-6}$, and
$t_{3}=1-3\cdot10^{-6}$. This indicates that if we move uniformly from $\Bw\in
V(\Bz)$ to $\Tld{\Bw}\in\Md V(\Bz)$, the admissible sets $\CA_{t}$ remain
virtually unchanged until we get very close to $\Md V(\Bz)$. The admissible
set then collapses rather abruptly to a point corresponding to
$\Tld{\Bw}\in\Md V(\Bz)$. This complicated, almost discontinuous behavior
occurs as we move from $\Bw$ to $\Tld{\Bw}$, which can barely be
distinguished in right panel of Figure~\ref{fig:lenses}.

The computations needed to make Figure~\ref{fig:lenses} have been done with the
Advanpix Multiprecision Computing Toolbox for MATLAB (www.advanpix.com) using
100 digits of precision.

\subsection{Interpolation}
\label{sub:intalg}
Let us assume now that the data $(\Bz,\Bw)\in\bb{C}^{2n}$ satisfies conditions
of Theorem~\ref{th:NP}, i.e., $\Bw\in V(\Bz)$. Our goal is to construct an interpolant $f\in\mathfrak{S}$,
such that $f(z_{j})=w_{j}$ for all $j=1,\ldots,n$. We begin with the case
$n=1$. According to  Theorem~\ref{th:NP}, the necessary and sufficient
condition for existence of such a function is $\im(w_{1})\ge 0$ and
$\im(z_{1}w_{1})\ge 0$. Of course, if $\im(w_{1})=0$, then $w_{1}\ge 0$,
according to the second inequality, and $f(z)=w_{1}$ for all $z$. If
$\im(z_{1}w_{1})=0$, then $zf(z)$ must be a real constant, and hence,
according to the first inequality, $f=-\Gs/z$, where $\Gs=-z_{1}w_{1}\ge
0$. Let us now assume that 
\begin{equation}
  \label{strictNP1}
\im(w_{1})>0,\qquad\im(z_{1}w_{1})>0,
\end{equation}
and characterize the set 
\[
\mathfrak{S}(z_{1},w_{1})=\{f\in\mathfrak{S}:f(z_{1})=w_{1}\}.
\]
We look for the answer in the same form as in the case of polynomials $\CP$,
where the set $\CP(z_{1},w_{1})$ of all polynomials $p\in\CP$ satisfying
$p(z_{1})=w_{1}$ can be described as
\[
\CP(z_{1},w_{1})=\{p(z)=(z-z_{1})q(z)+w_{1}: q\in\CP\}.
\] 
Moreover, distinct polynomials $q\in\CP$ correspond to distinct polynomials
$p\in\CP(z_{1},w_{1})$. By analogy with polynomials, we want to parametrize
the set $\mathfrak{S}(z_{1},w_{1})$ by elements of $\mathfrak{S}$ in the same
fashion as $\CP(z_{1},w_{1})$ is parametrized by elements of $\CP$. Of course,
we expect that the parametrization will be more complicated than in the case
of polynomials. The desired parametrization has already been found in
\cite{krnu98}, but the derivation here is not a routine calculation, differing
from the one in \cite{krnu98}.

According to Theorem~\ref{th:NP} the set of all admissible values $f(z)$ for
$f\in\mathfrak{S}(z_{1},w_{1})$ is described by the inequalities
\begin{equation}
  \label{DN}
      \det\BN([z_{1},z],[w_{1},f(z)])=
\dfrac{\im(w_{1})}{\im(z_{1})}\dfrac{\im f(z)}{\im(z)}
-\left|\dfrac{f(z)-\bra{w_{1}}}{z-\bra{z_{1}}}\right|^{2}\ge 0,
\end{equation}
\begin{equation}
  \label{DP}
\det\BP([z_{1},z],[w_{1},f(z)])=
\dfrac{\im(z_{1}w_{1})}{\im(z_{1})}\dfrac{\im (zf(z))}{\im(z)}
-\left|\dfrac{zf(z)-\bra{z_{1}w_{1}}}{z-\bra{z_{1}}}\right|^{2}\ge 0.
\end{equation}
Inequalities (\ref{DN}), (\ref{DP}) place $f(z)$ inside closed disks
$D_{N}(z_{1},w_{1},z)$ and $D_{P}(z_{1},w_{1},z)$, respectively.  At the same
time, Theorem~\ref{th:SviaR} says that $f\in\mathfrak{S}$ \IFF $f(z)$ lies in the
intersection of two closed half-planes
$\bra{\bb{H}}_{+}=\{w\in\bb{C}:\im(w)\ge 0\}$ and
$\bb{H}_{z}=\{w\in\bb{C}:\im(zw)\ge 0\}$, for every $z\in\bb{H}_{+}$. This
gives the idea of the desired parametrization of $\mathfrak{S}(z_{1},w_{1})$
by elements of $\mathfrak{S}$. This idea is at the core of the so-called V. Potapov's
method of ``fundamental matrix inequalities'' \cite{bosa99}. It has been implemented for
interpolation of matrix-valued Stieltjes functions in \cite{dyuka86}. We
present the argument and calculation both for the sake of completeness and
because the formulas here are used in our algorithm.

For every $z\in\bb{H}_{+}$ there
exists\footnote{Unique modulo $w\mapsto\Ga w$, $\Ga>0$ and $w\mapsto
  -1/(zw)$.} a fractional-linear transformation
\[
T_{z_{1},w_{1},z}(w)=\frac{L_{11}(z)w+L_{12}(z)}{L_{21}(z)w+L_{22}(z)}
\]
that maps $\CA(z;z_{1},w_{1})=D_{N}(z_{1},w_{1},z)\cap D_{P}(z_{1},w_{1},z)$
bijectively onto $\bra{\bb{H}}_{+}\cap\bb{H}_{z}$. In order to derive the
formula for $T_{z_{1},w_{1},z}(w)$ we exploit the simplicity of Stieltjes
functions corresponding to the points on the boundary of the admissible
regions $\CA(z;z_{1},w_{1})$ and $\bra{\bb{H}}_{+}\cap\bb{H}_{z}$. The idea is
that while the set of functions in $\mathfrak{S}(z_{1},w_{1})$ is very large,
if (\ref{strictNP1}) is satisfied, it degenerates to a single point if any of
the inequalities in (\ref{strictNP1}) become equalities, as we have already
discussed. The same holds for inequalities in (\ref{DN}), (\ref{DP}). If we
have equality in (\ref{DN}), then there exist a nonzero vector
$\BGx=(\xi_{1},\xi_{2})\in\ker\BN([z_{1},z_{2}],[f(z_{1}),f(z_{2})])$, where
for convenience of notation we replaced $z$ with $z_{2}$. Using representation
(\ref{Stielrep}), we compute
\[
\frac{f(z_{j})-\bra{f(z_{k})}}{z_{j}-\bra{z_{k}}}=\int_{0}^{\infty}\frac{d\Gs(t)}{(t-z_{j})(t-\bra{z_{k}})},\quad j,k=1,2.
\]
Thus,
\[
0=(\BN([z_{1},z_{2}],[f(z_{1}),f(z_{2})])\BGx,\BGx)_{\bb{C}^{2}}=\int_{0}^{\infty}\left|\frac{\xi_{1}}{t-z_{1}}+
\frac{\xi_{2}}{t-z_{2}}\right|^{2}d\Gs(t).
\]
This means that there is a non-zero vector $(\xi_{1},\xi_{2})\in\bb{C}^{2}$,
such that the function
\[
\phi(t)=\frac{\xi_{1}}{t-z_{1}}+\frac{\xi_{2}}{t-z_{2}}
\]
is identically zero on the support of $\Gs$. Since $z_{1}\not=z_{2}$ we
conclude that the support of $\Gs$ must be a single point, and the
corresponding Stieltjes function must have the form
\begin{equation}
  \label{dDN}
  f(z)=\Gg+\frac{\Gs}{t-z}.
\end{equation}
 Conversely, if the
spectral measure of $f\in\mathfrak{S}(w_{1},z_{1})$ is supported on a single point, then
we have equality in (\ref{DN}) for any $z\in\bb{H}_{+}$.

A similar analysis can be done for the case of equality in (\ref{DP}):
\[
0=(\BP([z_{1},z_{2}],[f(z_{1}),f(z_{2})])\BGx,\BGx)_{\bb{C}^{2}}=\Gg|\xi_{1}+\xi_{2}|^{2}
+\int_{0}^{\infty}\left|\frac{\xi_{1}}{t-z_{1}}+\frac{\xi_{2}}{t-z_{k}}\right|^{2}td\Gs(t).
\]
This equality implies that $f(z)$ must have either of two forms
\begin{equation}
  \label{dDP}
  f(z)=\Gg-\frac{\Gs_{0}}{z},\text{ or }f(z)=-\frac{\Gs_{0}}{z}+\frac{\Gs_{1}}{t_{1}-z}.
\end{equation}
We can regard the first form of $f(z)$ as a limit of the second one when
$\Gs_{1}=\Gg t_{1}$, as $t_{1}\to+\infty$. 

Now, since the fractional-linear transformation $T_{z_{1},w_{1},z}(w)$ maps the boundary of
$\CA(z;z_{1},w_{1})$ onto the boundary of $\bb{H}_{+}\cap\bb{H}_{z}$, the set
\[
S_{N}(z_{1},w_{1})=\{f\in\mathfrak{S}(z_{1},w_{1}):\det\BN([z_{1},z_{2}],[f(z_{1}),f(z_{2})])=0\},
\]
consisting of functions (\ref{dDN})
must be mapped by $T_{z_{1},w_{1},z}$ onto the set $\{g\in\mathfrak{S}:\im(g(z))=0\}$,
while the set
\[
S_{P}(z_{1},w_{1})=\{f\in\mathfrak{S}(z_{1},w_{1}):\det\BP([z_{1},z_{2}],[f(z_{1}),f(z_{2})])=0\},
\] 
consisting of functions (\ref{dDP}) must be mapped by $T_{z_{1},w_{1},z}$ onto the set
$\{g\in\mathfrak{S}:\im(zg(z))=0\}$. This gives us the desired equations. If we write
$g(z)=T_{z_{1},w_{1},z}(f(z))$, then
\begin{equation}
  \label{FviaM}
  f(z)=\frac{L_{22}(z)g(z)-L_{12}(z)}{L_{11}(z)-g(z)L_{21}(z)}.
\end{equation}
Hence, the coefficients $L_{ij}(z)$ must satisfy the following properties: for
any $\mu\ge 0$ the function $g(z)=\mu$ must be mapped into an element of
$S_{N}(z_{1},w_{1})$, i.e., function of the form (\ref{dDN}), while for any
$\nu\ge 0$ the function $g(z)=-\nu/z$ must be mapped to an element of
$S_{P}(z_{1},w_{1})$, i.e., function of the form (\ref{dDP}). We therefore
obtain the following system of equations for the unknown coefficients $L_{ij}(z)$:
\begin{equation}
  \label{NPsystem}
  \begin{cases}
    \dfrac{L_{22}(z)\mu-L_{12}(z)}{L_{11}(z)-\mu L_{21}(z)}=
\Gg(\mu)+\frac{\Gs(\mu)}{t(\mu)-z},\\[3ex]
    -\dfrac{L_{22}(z)\nu+zL_{12}(z)}{zL_{11}(z)+\nu L_{21}(z)}=
    -\dfrac{\Gs_{0}(\nu)}{z}+\dfrac{\Gs_{1}(\nu)}{t_{1}(\nu)-z},\\[3ex]
\dfrac{L_{22}(z_{1})\mu-L_{12}(z_{1})}{L_{11}(z_{1})-\mu L_{21}(z_{1})}=w_{1},\\[3ex]
-\dfrac{L_{22}(z_{1})\nu+z_{1}L_{12}(z_{1})}{z_{1}L_{11}(z_{1})+\nu L_{21}(z_{1})}=w_{1}.
  \end{cases}
\end{equation}
The last two equations are easy to solve, since the coefficients $L_{ij}(z)$
do not depend neither on $\mu$ nor on $\nu$. Thus, we must require that
\begin{equation}
  \label{Lpmdef}
  \begin{cases}
  L_{11}(z_{1})w_{1}+L_{12}(z_{1})=0,\\
  L_{21}(z_{1})w_{1}+L_{22}(z_{1})=0.
  \end{cases}
\end{equation}
In order to solve the other two equations we first observe that equations
\begin{equation}
  \label{constr}
  \begin{cases}
      \phi_{N}(z_{1})=\Gg+\dfrac{\Gs}{t-z_{1}}=w_{1},\\[2ex]
      \phi_{P}(z_{1})=-\dfrac{\Gs_{0}}{z_{1}}+\dfrac{\Gs_{1}}{t_{1}-z_{1}}=w_{1}
  \end{cases}
\end{equation}
determine two 1-parameter families of solutions $\phi_{N}(z;t)$ and
$\phi_{P}(z;t_{1})$, tracing the boundaries of $D_{N}$ and $D_{P}$,
respectively. Explicitly, we find
\begin{equation}
  \label{paramvals}
  \begin{cases}
\Gs=\dfrac{\im(w_{1})}{\im(z_{1})}|t-z_{1}|^{2},\\[3ex]
\Gg=\dfrac{\im(z_{1}w_{1})}{\im(z_{1})}-t\dfrac{\im(w_{1})}{\im(z_{1})},\\[3ex]
 \Gs_{0}=\left(\dfrac{\im(w_{1})}{\im(z_{1})}-
\nth{t_{1}}\dfrac{\im(z_{1}w_{1})}{\im(z_{1})}\right)|z_{1}|^{2},\\[3ex]
\Gs_{1}=\dfrac{|t_{1}-z_{1}|^{2}}{t_{1}}\dfrac{\im(z_{1}w_{1})}{\im(z_{1})}.
  \end{cases}
\end{equation}
This shows that for functions $\phi_{N}(z)$ and $\phi_{P}(z)$ to be in $\mathfrak{S}$
it is necessary and sufficient that $t\in[0,t_{*}]$ and
$t_{1}\in[t_{*},\infty]$, where
\[
t_{*}=\frac{\im(z_{1}w_{1})}{\im(w_{1})}.
\]
We then see that when $t=t_{1}=t_{*}$ we have
\begin{equation}
  \label{EE}
  \phi_{N}(z)=\phi_{P}(z)=\frac{\Gs_{*}}{t_{*}-z},\quad
\Gs_{*}=\dfrac{\im(w_{1})}{\im(z_{1})}|t_{*}-z_{1}|^{2}=\frac{|w_{1}|^{2}\im(z_{1})}{\im(w_{1})},
\end{equation}
while when $t=0$ and $t_{1}=\infty$
\begin{equation}
  \label{NEE}
  \phi_{N}(z)=\phi_{P}(z)=\Gg_{*}-\frac{\Gs^{*}}{z},\quad
\Gg_{*}=\dfrac{\im(z_{1}w_{1})}{\im(z_{1})},\quad
\Gs^{*}=\frac{|z_{1}|^{2}\im(w_{1})}{\im(z_{1})}.
\end{equation}
The correspondence between the two points of intersection of $\Md D_{N}$ and
$\Md D_{P}$ and the two points of intersection of $\Md\bb{H}_{+}$ and
$\Md\bb{H}_{z}$, characterized by $\mu=\nu=0$ and $\mu=\nu=\infty$,
respectively, is determined unambiguously by the orientation-preserving
property of fractional-linear transformations. We conclude that the
point $t=t_{1}=t_{*}$ corresponds to $\mu=\nu=0$, while the point
$\mu=\nu=\infty$ corresponds to $t=0$, $t_{1}=\infty$. Hence, we have the
equations
\[
-\dfrac{L_{12}(z)}{L_{11}(z)}=\frac{\Gs_{*}}{t_{*}-z},\qquad
-\dfrac{L_{22}(z)}{L_{21}(z)}=\Gg_{*}-\frac{\Gs^{*}}{z},
\]
that permit us to eliminate $L_{11}$ and $L_{22}$.
Denoting
$\Psi(z)=L_{21}(z)/L_{12}(z)$,
we obtain from the first equation in (\ref{NPsystem})
\[
  \frac{\left(\frac{\Gs^{*}}{z}-\Gg_{*}\right)\mu\Psi(z)-1}
{\frac{(z-t_{*})}{\Gs_{*}}-\mu\Psi(z)}=
\Gg(t(\mu))+\frac{\Gs(t(\mu))}{t(\mu)-z},\quad
\Gg(t)=\frac{\Gg_{*}}{t_{*}}(t_{*}-t),\quad
\Gs(t)=\frac{\Gg_{*}}{t_{*}}|t-z_{1}|^{2}
\]
Solving this equation for $\Psi$ (on Maple) %see Milton_extremals2.mw
we obtain that $\Psi(z)/z$ is a
ratio of two quadratic polynomials in $z$ with
\[
\lim_{z\to\infty}\frac{\Psi(z)}{z}=\frac{t(\mu)-t_{*}}{\Gs_{*}\mu t(\mu)}.
\]
Since $\Psi(z)$ does not depend on $\mu$ we conclude that there exists
$\Ga\in\bb{R}$, such that
\begin{equation}
  \label{tofmu}
  t(\mu)=\frac{t_{*}}{1-\Ga\Gs_{*}\mu}.
\end{equation}
Then, substituting (\ref{tofmu}) together with the parameter values
\begin{equation}
  \label{paramvals0} 
  t_{*}=\frac{\im(z_{1}w_{1})}{\im(w_{1})},\quad
\Gs_{*}=\frac{|w_{1}|^{2}\im(z_{1})}{\im(w_{1})},\quad
\Gg_{*}=\dfrac{\im(z_{1}w_{1})}{\im(z_{1})},\quad
\Gs^{*}=\frac{|z_{1}|^{2}\im(w_{1})}{\im(z_{1})}
\end{equation}
into the formula for $\Psi(z)$ in Maple we obtain that $\Psi(z)=\Ga z$.
We can now go back and recover the formulas for all of the coefficients $L_{ij}(z)$:
\[
\frac{L_{11}(z)}{L_{12}(z)}=\frac{z-t_{*}}{\Gs_{*}},\qquad
\frac{L_{22}(z)}{L_{12}(z)}=\Ga(\Gs^{*}-\Gg_{*}z).
\]
In this case it is easy to see that equations (\ref{Lpmdef}) will be
satisfied. Thus, the desired fractional-linear transformation is given by
\begin{equation}
  \label{MviaFalpha}
  T_{z_{1},w_{1},z}(f(z))=g(z)=\nth{\Ga}\cdot\frac{(z-t_{*})f(z)+\Gs_{*}}{z f(z)+\Gs^{*}-\Gg_{*}z},
\end{equation}
where the sign of $\Ga$ needs to be determined. It is easy to do when we
examine the behavior of functions $f(z)$ and $g(z)$ at infinity.
If we define
\[
\Gg_{g}=\lim_{z\to\infty}g(z),\qquad\Gg_{f}=\lim_{z\to\infty}f(z),
\]
then, according to (\ref{MviaFalpha})
\[
  \Gg_{g}=\frac{\Gg_{f}}{\Ga(\Gg_{f}-\Gg_{*})},\qquad\Gg_{f}=\frac{\Ga\Gg_{*}\Gg_{g}}{\Ga\Gg_{g}-1}.
\]
Since for any $g\in\mathfrak{S}$ we must get $f\in\mathfrak{S}(z_{1},w_{1})$
we conclude that we must have $\Ga<0$. Since multiplication by $-\Ga>0$ maps the
intersection of the two half-planes $\bb{H}_{+}$ and $\bb{H}_{z}$ onto itself,
any choice of $\Ga<0$ will produce a valid parametrization of
$f\in\mathfrak{S}(z_{1},w_{1})$ by $g\in\mathfrak{S}$. For simplicity we set
$\Ga=-1$ and obtain the desired
parametrization of $\mathfrak{S}(z_{1},w_{1})$:
\begin{equation}
  \label{FviaM0}
  \mathfrak{S}(z_{1},w_{1})=\left\{f(z)=\frac{g(z)(\Gg_{*}z-\Gs^{*})-\Gs_{*}}{zg(z)+z-t_{*}}:g\in\mathfrak{S}\right\},
\end{equation}
where the parameters $\Gg_{*}$, $\Gs^{*}$, $\Gs_{*}$, and $t_{*}$ are given in
(\ref{paramvals0}), and provided inequalities (\ref{strictNP1}) hold. The
exact same formula (but with different normalization for $g(z)$) has been
obtained\footnote{There is a typo in \cite{krnu98}: $a_{12}^{(1)}$ should be
$|c_{1}|^{2}/w_{1}$.} in \cite{krnu98}.

The parametrization (\ref{FviaM0}) has several useful properties. At
infinity we obtain
\begin{equation}
  \label{gammas}
  \Gg_{f}=\frac{\Gg_{*}\Gg_{g}}{\Gg_{g}+1}.
\end{equation}
This can be important in
applications in the context of complex electromagnetic susceptibility functions,
where the physically mandated assumption on the interpolant $f\in\mathfrak{S}$ is
$\Gg_{f}=0$. Formula (\ref{gammas}) shows that $\Gg_{f}=0$ \IFF
$\Gg_{g}=0$. This means that starting with $g(z)=0$ and iterating formula
(\ref{FviaM0}) will always result in a decaying Stieltjes function $f(z)$.

Another nice feature of (\ref{FviaM0}) is the degree-reduction property.
To exhibit it let us solve (\ref{FviaM0}) for $g(z)$:
\begin{equation}
  \label{Mviaf}
  g(z)=\frac{f(z)(t_{*}-z)-\Gs_{*}}{zf(z)+\Gs^{*}-\Gg_{*} z}.
\end{equation}
\begin{theorem}
\label{th:degred}
  Suppose $f\in\mathfrak{S}(z_{1},w_{1})$ is a rational function of degree
  $n\ge 1$ in the sense that $f(z)=P_{n}(z)/Q_{n}(z)$, where the degree of
  $Q_{n}$ is $n$, while the degree of $P_{n}$ is either $n$ or $n-1$, while
  $P_{n}$ and $Q_{n}$ have no common roots. Then $g(z)$, given by
  (\ref{Mviaf}) is a rational function in $\mathfrak{S}$ of degree $n-1$ in
  the same sense as above.
\end{theorem}
\begin{proof}
  The essential feature of (\ref{Mviaf}) is that all of its
  coefficients $L_{ij}(z)$ are linear in $z$.
If $f(z)=P_{n}(z)/Q_{n}(z)$, then
\[
L_{11}(z)f(z)+L_{12}(z)=\frac{L_{11}(z)P_{n}(z)+L_{12}(z)Q_{n}(z)}{Q_{n}(z)}.
\]
Formulas (\ref{Lpmdef}) imply that the polynomial
$L_{11}(z)P_{n}(z)+L_{12}(z)Q_{n}(z)$ will have a pair of complex conjugate
roots $z_{1}$ and $\bra{z_{1}}$. We can therefore write
\[
L_{11}(z)P_{n}(z)+L_{12}(z)Q_{n}(z)=(z-z_{1})(z-\bra{z_{1}})T^{+}(z),\qquad
L_{11}=t_{*}-z,\quad L_{12}=-\Gs_{*}.
\]
It follows that $\deg(T^{+})=n-1$, if $\deg(P_{n})=n$, and $\deg(T^{+})\le
n-2$, if $\deg(P_{n})=n-1$. Similarly,
\[
L_{21}(z)P_{n}(z)+L_{22}(z)Q_{n}(z)=(z-z_{1})(z-\bra{z_{1}})T^{-}(z),\qquad
L_{21}=z,\quad L_{22}=\Gs^{*}-\Gg_{*} z,
\]
and $\deg(T^{-})\le n-1$ , if $\deg(P_{n})=n$, and $\deg(T^{-})=n-1$, if
$\deg(P_{n})=n-1$. Since $g(z)=T^{+}(z)/T^{-}(z)$ is in $\mathfrak{S}$ the
degree of $T^{-}(z)$ can be at most one above the degree of $T^{+}(z)$. This
shows that we can only have equalities in the degree inequalities
above. Finally, if $T^{+}$ and $T^{-}$ have common roots, then formula
(\ref{FviaM0}) would imply that $f(z)$ is a rational function of degree
strictly less than $n$, contradicting our assumption. 
\end{proof}

The parametrization (\ref{FviaM0}) of $\mathfrak{S}(z_{1},w_{1})$ by
elements of $\mathfrak{S}$ leads to the recursive interpolation algorithm.
Given the data $\Bw\in V(\Bz)$, $\Bz=(z_{1},\ldots,z_{n})$ for $n$ distinct points
$\{z_{1},\ldots,z_{n}\}\subset\bb{H}_{+}$, we define the interpolant $f(z)$ by
(\ref{FviaM0}), where $g(z)\in\mathfrak{S}$ satisfies $n-1$ constraints
\begin{equation}
  \label{recursion}
g(z_{j})=\frac{L_{11}(z_{j})w_{j}+L_{12}(z_{j})}{w_{j}L_{21}(z_{j})+L_{22}(z_{j})},
\quad j=2,\ldots,n,
\end{equation}
provided
\[
w_{j}L_{21}(z_{j})+L_{22}(z_{j})\not=0,\qquad j=2,\ldots,n.
\]
In that case $f(z_{j})=w_{j}$,
$j=2,\ldots,n$, and
\[
f(z_{1})=\frac{L_{22}(z_{1})g(z_{1})-L_{12}(z_{1})}{L_{11}(z_{1})-g(z_{1})L_{21}(z_{1})}.
\]
Using equations (\ref{Lpmdef}) we obtain
\[
f(z_{1})=\frac{-L_{21}(z_{1})w_{1}g(z_{1})+L_{11}(z_{1})w_{1}}{L_{11}(z_{1})-g(z_{1})L_{21}(z_{1})}
=w_{1},
\]
provided $L_{11}(z_{1})-g(z_{1})L_{21}(z_{1})\not=0$. This condition is always
satisfied, since linear functions $L_{ij}(z)$ are such that
$f\in\mathfrak{S}$, for any $g\in\mathfrak{S}$. This requires that the denominator in
(\ref{FviaM}) never vanishes when $z\in\bb{H}_{+}$. 

In order to finish the analysis we need to consider
the special case when there exists $k\in\{2,\ldots,n\}$, such that
\begin{equation}
  \label{exitearly}
  L_{22}(z_{k})+w_{k}L_{21}(z_{k})=0.
\end{equation}
In this case the corresponding relation (\ref{recursion}) will be
undefined. But in this case the four real equations
\[
\begin{cases}
  L_{22}(z_{k})+w_{k}L_{21}(z_{k})=0,\\
  L_{22}(z_{1})+w_{1}L_{21}(z_{1})=0
\end{cases}
\]
form a linear homogeneous system of equations with four real unknowns $a_{21}$,
$a_{22}$, $b_{21}$, and $b_{22}$, where
\[
L_{21}(z)=a_{21}z+b_{21},\qquad L_{22}(z)=a_{22}z+b_{22}.
\]
 Thus, the determinant of this system must
vanish. Maple calculations % see Milton_extremals2.mw
show that this implies that $\det\BN=0$, where
\[
\BN=
\begin{bmatrix}
\dfrac{w_{1}-\bra{w_{1}}}{z_{1}-\bra{z_{1}}}&\dfrac{w_{1}-\bra{w_{k}}}{z_{1}-\bra{z_{k}}}\\[2ex]
\dfrac{w_{k}-\bra{w_{1}}}{z_{k}-\bra{z_{1}}}&\dfrac{w_{k}-\bra{w_{k}}}{z_{k}-\bra{z_{k}}}
\end{bmatrix}.
\]
We have already proved that in this case the support of $\Gs$ must be a single
point. Thus, when (\ref{exitearly}) is satisfied we just return the rational
function $\phi_{N}(z)$, given by (\ref{NEE}). Indeed, (\ref{exitearly}) implies
\[
w_{k}=\Gg_{*}-\frac{\Gs^{*}}{z_{k}}=\phi_{N}(z_{k}).
\]
At the same time $\phi_{N}(z)$ also satisfies
$\phi_{N}(z_{1})=w_{1}$. It follows that $f(z)=\phi_{N}(z)$.

\subsection{The least squares problem}
For $\Bw\in\bb{C}^{n}$ there are two mutually exclusive logical
possibilities. Either $\Bw\in V(\Bz)$ or $\Bw\not\in V(\Bz)$. The former case,
called the \emph{interpolation problem} has been considered in the previous
section. In the latter case, when there is no Stieltjes function $f$ satisfying
$\Bw=f(\Bz)$, we want to solve the \emph{least squares problem}
(\ref{lsqStielt}), which can be also reformulated as
\begin{equation}
  \label{lsqV}
  \GS(\Bw,\Bz)=\min_{\Bp\in V(\Bz)}|\Bp-\Bw|.
\end{equation}
The minimizer $\Bp^{*}$ of (\ref{lsqV}) exists because $V(\Bz)$ is a closed
subset of $\bb{C}^{n}$. It is unique because $V(\Bz)$ is convex. Moreover,
since $\Bw\not\in V(\Bz)$, the minimizer $\Bp^{*}$ must lie on the boundary of
$V(\Bz)$. In this case, the Nevanlinna-Pick theorem~\ref{th:NP} for the
Stieltjes class says that there exists a unique Stieltjes function
$f_{*}\in\mathfrak{S}$ satisfying $f_{*}(\Bz)=\Bp^{*}$.

Let us analyze the properties of this unique minimizer. Here we follow the
analysis of Caprini \cite{capr80}, who derived the necessary and sufficient
conditions for a minimizer in (\ref{lsqV}). Caprini's method is based on our
ability to compute the effect of variations of $\Gg$ and spectral
measure $\Gs$ in representation (\ref{Stielrep}) on the
functional we want to minimize. Suppose that
\[
f_{*}(z)=\Gg+\int_{0}^{\infty}\frac{d\Gs(t)}{t-z}
\]
is the minimizer in (\ref{lsqStielt}). Then $p_{j}^{*}=f_{*}(z_{j})$ minimizes
(\ref{lsqV}).  Let
\begin{equation}
  \label{compet}
  \Tld{f}(z)=\Tld{\Gg}+\int_{0}^{\infty}\frac{d\Tld{\Gs}(t)}{t-z}
\end{equation}
be a competitor in (\ref{lsqStielt}), and let $\Tld{p}_{j}=\Tld{f}(z_{j})$. The variation
$\phi=\Tld{f}-f_{*}$ can then be written as
\[
\phi(z)=\GD\Gg+\int_{0}^{\infty}\frac{d\nu(t)}{t-z},\qquad
\nu=\Tld{\Gs}-\Gs,\quad\GD\Gg=\Tld{\Gg}-\Gg.
\]
We then compute
\[
|\Tld{\Bp}-\Bw|^{2}-|\Bp^{*}-\Bw|^{2}%=|(\Tld{\Bp}-\Bp^{*})+\Bp^{*}-\Bw|^{2}-|\Bp^{*}-\Bw|^{2}
=|\Tld{\Bp}-\Bp^{*}|^{2}+2\re(\Bp^{*}-\Bw,\Tld{\Bp}-\Bp^{*}).
\]
Observing that
\[
\Tld{p}_{j}-p^{*}_{j}=\GD\Gg+\int_{0}^{\infty}\frac{d\nu(t)}{t-z_{j}},
\]
we see that
\[
\re(\Bp^{*}-\Bw,\Tld{\Bp}-\Bp^{*})=(\GD\Gg)\re\sum_{j=1}^{n}(p^{*}_{j}-w_{j})+
\int_{0}^{\infty}\re\sum_{j=1}^{n}\frac{p^{*}_{j}-w_{j}}{t-\bra{z_{j}}}d\nu(t).
\]
The real rational function
\begin{equation}
  \label{Caprini}
C(t)=\re\sum_{j=1}^{n}\frac{p^{*}_{j}-w_{j}}{t-\bra{z_{j}}},\qquad  t\ge 0,
\end{equation}
which we will call \emph{the Caprini function}, will play an essential role in our
algorithm for solving the least squares problem (\ref{lsqStielt}).

In terms of the Caprini function we obtain
\begin{equation}
  \label{var}
    |\Tld{\Bp}-\Bw|^{2}-|\Bp^{*}-\Bw|^{2}=
2(\GD\Gg)\lim_{t\to\infty}tC(t)+2\int_{0}^{\infty}C(t)d\nu(t)+|\Tld{\Bp}-\Bp^{*}|^{2}.
\end{equation}
This formula permits us to formulate and prove Caprini's necessary and
sufficient conditions for the minimizer in (\ref{lsqV}). This is a particular
version of Caprini's result \cite{capr80}, where the real and imaginary parts
of each individual measurement could have a different weight in the least
squares functional.
\begin{theorem}
  \label{th:Caprini}
Suppose that the minimum in (\ref{lsqStielt}) is nonzero, then the
unique minimizer $f_{*}\in\mathfrak{S}$ is given by
 \begin{equation}
   \label{fstar}
   f_{*}(z)=\Gg+\sum_{j=1}^{N}\frac{\Gs_{j}}{t_{j}-z}
 \end{equation}
 for some $\Gs_{j}>0$, $t_{j}\ge 0$ and $\Gg\ge 0$.  Moreover, $f_{*}$, given
 by (\ref{fstar}) is the minimizer in (\ref{lsqStielt}) \IFF
its Caprini
 function $C(t)$ is nonnegative and vanishes at $t=t_{j}$, $j=1,\ldots,N$, and
``at infinity'', in the sense that
\begin{equation}
  \label{zatinf}
  \re\sum_{j=1}^{n}(p_{j}^{*}-w_{j})=\lim_{t\to\infty}tC(t)=0,
\end{equation}
provided $\Gg>0$.
\end{theorem}
\begin{proof}
If $\Gg>0$, then we can consider the competitor (\ref{compet}) with
$\Tld{\Gs}=\Gs$. Formula (\ref{var}) then implies that
\[
2(\GD\Gg)\lim_{t\to\infty}tC(t)+(\GD\Gg)^{2}>0,
\]
where $\GD\Gg$ can be either positive or negative and can be chosen as small
in absolute value as we want. This implies (\ref{zatinf}).

Next, suppose $t_{0}\in[0,+\infty)$ is in the support of $\Gs$. For every
$\Ge>0$ we define $I_{\Ge}(t_{0})=\{t\ge 0:|t-t_{0}|<\Ge\}$. Saying that
$t_{0}$ is in the support of $\Gs$ is equivalent to
$m(t_{0},\Ge)=\Gs(I_{\Ge}(t_{0}))>0$ for all $\Ge>0$.  Then, there are two
possibilities. Either
  \begin{enumerate}
  \item[(i)] $\displaystyle\lim_{\Ge\to 0}m(t_{0},\Ge)=0$, or
  \item[(ii)] $\displaystyle\lim_{\Ge\to 0}m(t_{0},\Ge)=\Gs_{0}>0$
  \end{enumerate}
In case (i) we construct a competitor measure
\[
\Gs_{\Ge}=\Gs-\Gs|_{I_{\Ge}(t_{0})}+\Gth m(t_{0},\Ge)\Gd_{t_{0}},
\]
where $\Gth>0$ is an arbitrary constant. We then define
\begin{equation}
  \label{feps}
  f_{\Ge}(z)=\Gg+\int_{0}^{\infty}\frac{d\Gs_{\Ge}(t)}{t-z},\qquad
p^{\Ge}_{j}=f_{\Ge}(z_{j}).
\end{equation}
Formula (\ref{var}) then implies
\[
\lim_{\Ge\to 0}\frac{|\Bp^{\Ge}-\Bw|^{2}-|\Bp^{*}-\Bw|^{2}}{m(t_{0},\Ge)}=
2(\Gth-1)C(t_{0}),
\]
since $|\Bp^{\Ge}-\Bp^{*}|\le Cm(t_{0},\Ge)$, where $C$ is independent of $\Ge$.
If $f_{*}$ is the minimizer, then we must have $(\Gth-1)C(t_{0})\ge 0$ for all
$\Gth> 0$, which implies that $C(t_{0})=0$. 

In the case (ii) we have $\Gs(\{t_{0}\})=\Gs_{0}>0$. Then, for every
$|\Ge|<\Gs_{0}$ we construct a competitor measure
\begin{equation}
  \label{compmeas}
  \Gs_{\Ge}=\Gs+\Ge\Gd_{t_{0}},
\end{equation}
as well as the corresponding $f_{\Ge}$ and $\Bp^{\Ge}$, given by (\ref{feps}). We then compute
\begin{equation}
  \label{sigmavar}
  \lim_{\Ge\to 0}\frac{|\Bp^{\Ge}-\Bw|^{2}-|\Bp^{*}-\Bw|^{2}}{\Ge}=2C(t_{0}).
\end{equation}
Since in this case $\Ge$ can be both positive and negative we conclude that
$C(t_{0})=0$. 

Hence, we have shown that $C(t_{0})=0$ whenever $t_{0}\in[0,+\infty)$ is in
the support of the spectral measure $\Gs$ of the minimizer $f_{*}$. It remains
to observe that for any $t\in\bb{R}$
\[
C(t)=\sum_{j=1}^{n}\left\{\frac{p^{*}_{j}-w_{j}}{t-\bra{z_{j}}}+
\frac{\bra{p^{*}_{j}}-\bra{w_{j}}}{t-z_{j}}\right\}.
\]
Thus, $C(t)$ is a restriction to the real line of a rational function on the
\nbh\ of the real line in the complex $t$-plane. By assumption, $\Bw\not\in
V(\Bz)$, and therefore $C(t)$ is not identically zero. In particular, $C(t)$
cannot have more than $2n-1$ zeros. We conclude that the support of the
spectral measure of the minimizer $f_{*}$ must be finite, and the minimizer
must be a rational function.

Now let us consider the competitor (\ref{feps}) defined by (\ref{compmeas}),
where $\Ge>0$ and $t_{0}$ is not in the support of $\Gs$. Formula (\ref{sigmavar})
then implies that
\[
  \lim_{\Ge\to 0^{+}}\frac{|\Bp^{\Ge}-\Bw|^{2}-|\Bp^{*}-\Bw|^{2}}{\Ge}=2C(t_{0})\ge 0.
\]
This proves that $C(t)\ge 0$ for all $t\ge 0$. The necessity of the stated
properties of the Caprini function $C(t)$ is now established. 

Sufficiency is a direct consequence of formula (\ref{var}). For any competitor
measure $\Tld{\Gs}$ we can write
\[
\nu=\Tld{\Gs}-\Gs=\sum_{j=1}^{N}(\GD\Gs_{j})\Gd_{t_{j}}+\Tld{\nu},
\]
where $\Tld{\nu}$ is a positive Radon measure without any point masses at
$t=t_{j}$, $j=1,\ldots,N$. It is obtained by eliminating point masses of
$\Tld{\Gs}$ at $t_{j}$, $j=1,\ldots,N$, if it has any:
\[
\Tld{\nu}=\Tld{\Gs}-\sum_{j=1}^{N}\Tld{\Gs}(\{t_{j}\})\Gd_{t_{j}}.
\]
We then compute, via formula (\ref{var}), taking
into account that $C(t_{j})=0$
\[
|\Tld{\Bp}-\Bw|^{2}-|\Bp^{*}-\Bw|^{2}=
2(\GD\Gg)\lim_{t\to\infty}tC(t)+2\int_{0}^{\infty}C(t)d\Tld{\nu}(t)+|\Tld{\Bp}-\Bp^{*}|^{2}\ge 0,
\]
since $C(t)\ge 0$. If $\GD\Gg<0$, then $\Gg=\Tld\Gg-\GD\Gg>0$, and therefore
the first term on the \rhs\ vanishes due to (\ref{zatinf}).
\end{proof}
We observe that that if $t_{j}>0$, then we must also have $C'(t_{j})=0$, since
$t=t_{j}$ is a point of local minimum of $C(t)$. If we write formula (\ref{fstar}) in
the form
\[
f_{*}(z)=\Gg-\frac{\Gs_{0}}{z}+\sum_{j=1}^{N}
\frac{\Gs_{j}}{t_{j}-z},\qquad\Gg\ge 0,\ \Gs_{0}\ge 0,\ t_{j}>0,\
\Gs_{j}>0,\ j=1,\ldots,N,
\]
then we have exactly $2(N+1)$ equations for $2(N+1)$ unknowns $\Gg$,
$\Gs_{0}$, $t_{j}$, $\Gs_{j}$, $j=1,\ldots,N$:
\begin{equation}
  \label{Copteqs}
  \Gg\lim_{t\to\infty}tC(t)=0,\quad\Gs_{0}C(0)=0,\quad C(t_{j})=0,\quad
C'(t_{j})=0,\quad j=1,\ldots,N.
\end{equation}
Obviously, these equations do not enforce the nonnegativity of $C(t)$ and may
very well be satisfied when some $t_{j}$ are points of local maxima and $C(t)$
is not nonnegative. Hence, the equations should not really be regarded as
equations for the minimizer. Instead the intended use of
Theorem~\ref{th:Caprini} is to provide the certificate of optimality of a
purported solution of (\ref{lsqV}) by exhibiting the graph of $C(t)$ that
shows that the necessary and sufficient conditions of optimality are
satisfied. In fact, equations (\ref{Copteqs}) are used in our algorithm to
make the final adjustments when a near-optimal solution is obtained.

\subsection{Analytic structure of the boundary of $V(\Bz)$}
The analytic structure of the interpolation body $V(\Bz)$ defined in
(\ref{Vofzdef}) is well-understood. The set $V(\Bz)$ is a closed convex cone in
$\bb{C}^{n}$ with non-empty interior $V^{\circ}(\Bz)$, characterized by the
inequalities $\BN(\Bz,\Bw)>0$, $\BP(\Bz,\Bw)>0$ in the sense of quadratic forms. The
set 
\[
\mathfrak{S}(\Bz,\Bw)=\{f\in\mathfrak{S}:f(\Bz)=\Bw\}
\] 
is parametrized by
elements of $\mathfrak{S}$ via the recursive interpolation procedure described
in Section~\ref{sub:intalg}. The function $f\in\mathfrak{S}(\Bz,\Bw)$
corresponding to $0\in\mathfrak{S}$ in such a parametrization has the form
\[
f(z)=\sum_{j=1}^{n}\frac{\Gs_{j}}{t_{j}-z},\quad\Gs_{j}>0,\ 0<t_{1}<\dots<t_{n},
\]
with the list of parameters $\Gs_{j}$ and $t_{j}$ above, in one-to-one
correspondence with points $\Bw=f(\Bz)$ in $V^{\circ}(\Bz)$
\cite{mi81,milt81b}. 

By contrast with $V^{\circ}(\Bz)$, each point on $\Md V(\Bz)$ can be
realized as a list of values of a unique Stieltjes function, which must
necessarily be rational. In view of Theorem~\ref{th:NP} the boundary of
$V(\Bz)$ can be naturally written as a union of two overlapping sets
\[
\Md V^{N}(\Bz)=\{\Bw\in V(\Bz):\det\BN(\Bz,\Bw)=0\},\ 
\Md V^{P}(\Bz)=\{\Bw\in V(\Bz):\det\BP(\Bz,\Bw)=0\}.
\]
We can think of them as two sides of a clam shell that meet along the ``rim''
\[
\Md V^{NP}(\Bz)=\{\Bw\in V(\Bz):\det\BN(\Bz,\Bw)=0,\ \det\BP(\Bz,\Bw)=0\}.
\]
Each point $\Bw\in \Md V^{N}(\Bz)$ is attained by a unique rational
function $f\in \mathfrak{S}_{n}^{N}$, where
\begin{equation}
  \label{Nform}
\mathfrak{S}_{n}^{N}=\left\{\Gg+\sum_{k=1}^{n-1}\frac{\Gs_{k}}{t_{k}-z}: \Gg\ge
0,\ t_{k}\ge 0,\ \Gs_{k}\ge 0\right\}.
\end{equation}
Similarly, each point $\Bw\in \Md V^{P}(\Bz)$ is attained by a unique rational
function $f\in \mathfrak{S}_{n}^{P}$.  Unfortunately, a simple representation,
like (\ref{Nform}) of functions in $\mathfrak{S}_{n}^{P}$ is not
possible. This is because the parameter space $(\Gg,\BGs,\Bt)$ in
(\ref{Nform}) is non-compact, and it is an accident that the set
$\mathfrak{S}_{n}^{N}$ happens to be closed (in the space of holomorphic
functions on $\bb{C}\setminus\bb{R}_{+}$). The most concise, but somewhat
indirect description of $\mathfrak{S}_{n}^{P}$ can be formulated using the
``reflection'' symmetry $\CR: f\mapsto-1/(zf)$ of class $\mathfrak{S}$:
$\mathfrak{S}_{n}^{P}=\CR(\mathfrak{S}_{n}^{N})$.
Another description of $\mathfrak{S}_{n}^{P}$ is the closure of the set
\begin{equation}
  \label{Pform}
\Tld{\mathfrak{S}}_{n}^{P}=\left\{-\frac{\Gs_{0}}{z}+
\sum_{k=1}^{n-1}\frac{\Gs_{k}}{t_{k}-z}:\Gs_{0}\ge 0,\ t_{k}\ge 0,\ \Gs_{k}\ge 0\right\}
\end{equation}
with respect to the uniform convergence on compact subsets of
$\bb{C}\setminus\bb{R}_{+}$. Explicitly, the set $\mathfrak{S}_{n}^{P}$ can be
described as $\mathfrak{S}_{n}^{P}=\Tld{\mathfrak{S}}_{n}^{P}\cup\mathfrak{S}_{n-1}^{N}$.

Similarly, each point $\Bw\in \Md V^{NP}(\Bz)$ is
attained by a unique rational function $f\in\mathfrak{S}_{n}^{NP}$, where
$\mathfrak{S}_{n}^{NP}$ can be described implicitly as the closure of 
\begin{equation}
  \label{NPform}
  \Tld{\mathfrak{S}}_{n}^{NP}=\left\{\sum_{k=1}^{n-1}\frac{\Gs_{k}}{t_{k}-z}:
    t_{k}\ge 0,\ \Gs_{k}\ge 0\right\},
\end{equation}
or explicitly, as $\mathfrak{S}_{n}^{NP}=\Tld{\mathfrak{S}}_{n}^{NP}\cup\mathfrak{S}_{n-1}^{N}$.

If we define the evaluation operator $E_{\Bz}:\mathfrak{S}\to\bb{C}^{n}$ by
$E_{\Bz}f=f(\Bz)$, then we have both
\[
V(\Bz)=E_{\Bz}(\mathfrak{S})\text{ and } V(\Bz)=E_{\Bz}(\mathfrak{S}_{n+1}^{NP}).
\]
Moreover, $E_{\Bz}:\mathfrak{S}_{n+1}^{NP}\to V(\Bz)$ is a bijection. The
statements above are all consequences of the following classical theorem
\cite{kakr74,kackr74}.
\begin{theorem}
\label{th:prodrep}
  Suppose that $f\in\mathfrak{S}$ is a rational function. Then it can be written
  uniquely in the form
\begin{equation}
  \label{Frep}
  f(z)=\Gg+\sum_{j=1}^{n}\frac{\Gs_{j}}{t_{j}-z},\qquad\Gg\ge 0,\ \Gs_{j}>0,\
0\le t_{1}<t_{2}<\dots<t_{n},
\end{equation}
where $n\ge 0$ is an integer. If $\Gg>0$, then $f(z)$ has exactly $n$
distinct real zeros $z=x_{j}$, $j=1,\ldots,n$, satisfying the interlacing property
\[
0\le t_{1}<x_{1}<t_{2}<x_{2}<\ldots<t_{n}<x_{n}<+\infty,
\]
so that $f(z)$ can also be written as a product
\begin{equation}
  \label{prodrep}
  f(z)=\Gg\prod_{j=1}^{n}\frac{z-x_{j}}{z-t_{j}}.
\end{equation}
If $\Gg=0$ and $n\ge 1$, then there are exactly $n-1$ distinct real zeros
$z=x_{j}$ and
\begin{equation}
  \label{prodrep0}
  f(z)=\frac{A}{t_{n}-z}\prod_{j=1}^{n-1}\frac{z-x_{j}}{z-t_{j}},\quad A>0,\
0\le t_{1}<x_{1}<t_{2}<\ldots<x_{n-1}<t_{n}<+\infty.
\end{equation}
\end{theorem}

\section{A needle in a haystack}
\setcounter{equation}{0} 
\label{sec:needle}
In Section~\ref{sec:prelim} we have summarized a substantial body of existing
knowledge about the Stieltjes class $\mathfrak{S}$ and the closed convex cone
$V(\Bz)$. Can one harness this knowledge to devise an algorithm solving the
least squares problem (\ref{lsqV})?  Surprisingly the answer is not
apparent. What has been described so far is an interpolation algorithm for
constructing functions $f(z)$, satisfying $f(z_{j})=p^{*}_{j}$, once the
solution $\Bp^{*}$ of (\ref{lsqV}) has been found. In this section we take a
closer look at the geometry of $V(\Bz)$.
\begin{figure}[t]
  \centering
  \includegraphics[scale=0.3]{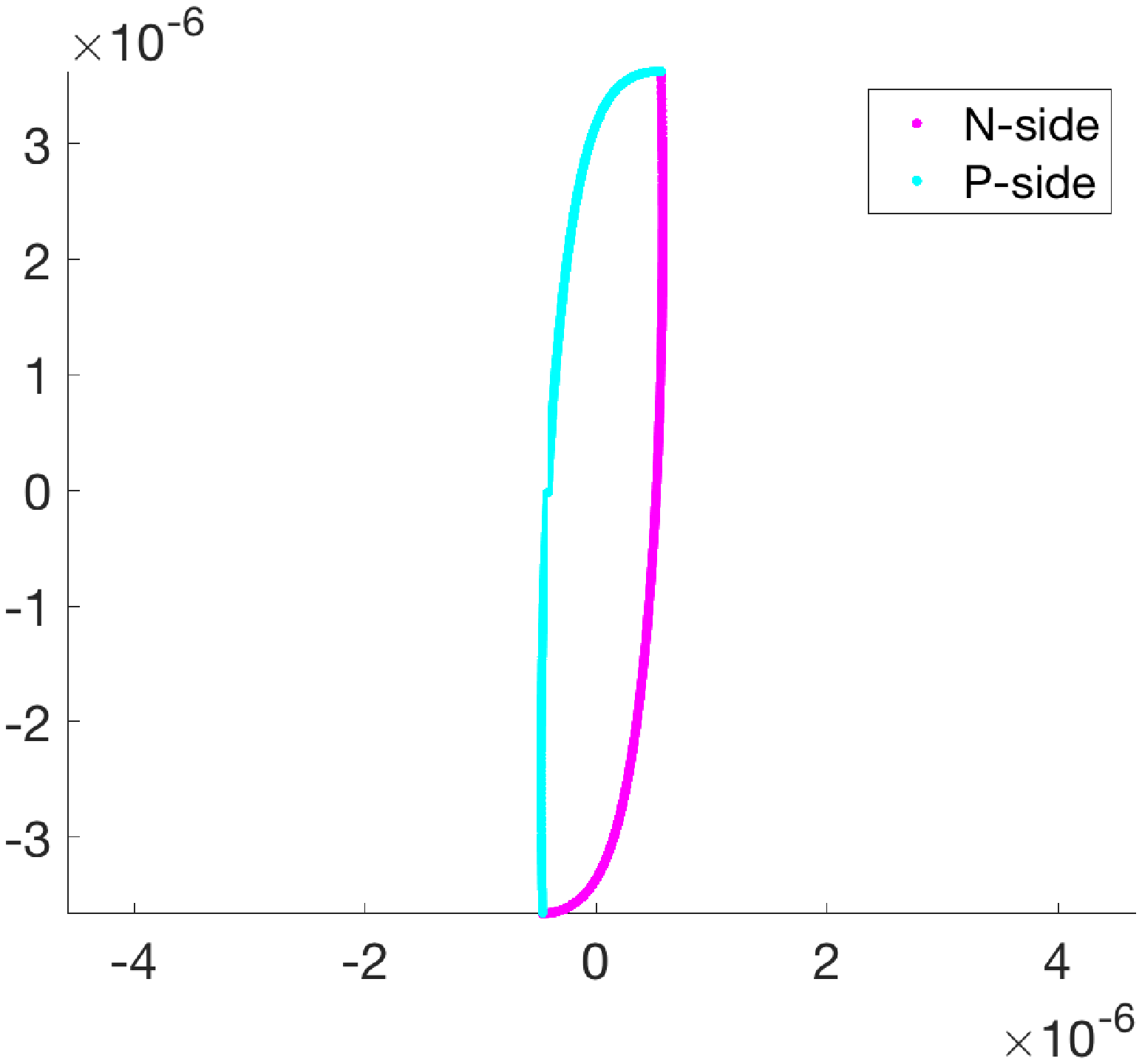}
% Made by rsch/matlab/kramers-kronig/Vofz_proj/projection/V_crossect.m
% loading data from w4Vcrossect.mat
  \caption{A random two-dimensional cross-section of $V(\Bz)$. The origin
    corresponds to $\Bw\in V(\Bz)$ and the cross-section is spanned by a random unit
    vector $\Bd\in\bb{C}^{n}$ and a normal $\Bn$ to $\Md V(\Bz)$ at the point
    where $\Bw+t\Bd$ intersects $\Md V(\Bz)$.}
  \label{fig:Vcrossect}
\end{figure}
Here will show that in effect, the set $V(\Bz)\subset\bb{C}^{n}$ has a very
small (real) dimension compared to $2n$. The proverbial needle analogy is apt
here. Even though the needle is a three-dimensional body, we can approximate
it well by an interval of a straight line. To illustrate our point we return
to our simple example (\ref{sqrtexample0}). Figure~\ref{fig:Vcrossect} shows a
two-dimensional cross-section of $V(\Bz)$, where $\Md V^{N}(\Bz)$ and $\Md
V^{P}(\Bz)$ parts of the boundary of $V(\Bz)$ are shown in magenta and cyan
and are on the left and the right side of $V(\Bz)$, respectively. The origin
in the figure is placed at $\Bw$ in the interior of $V(\Bz)$.  When we added a
2\% noise to $\Bw$, the noisy data $\Tld{\Bw}$ would lie about 25,000
thicknesses of the cross-section away. If an ordinary sawing needle is the
analogy for $V(\Bz)$, the point $\Tld{\Bw}$ would be about 25 meters away.

To see the dimensional degeneracy of $V(\Bz)$ mathematically we recall that
the rank-two displacement structure (\ref{ND}) and (\ref{PD}) of $\BN(\Bz,\Bw)$
and $\BP(\Bz,\Bw)$, respectively, implies that their eigenvalues decay exponentially
fast \cite{beto17}. Hence, numerically, these matrices will always have
eigenvalues which are indistinguishable from 0 up to the floating point
precision, when $n>15$. Thus, numerically, all points in $V(\Bz)$ will appear
to lie on its boundary.
 
The crucial point here is that the dimensional degeneracy of the geometry of
$V(\Bz)$ handily defeats typical minimization algorithms that start with some
initial guess $\Bp_{0}\in\Md V(\Bz)$ and choose the direction in which we want to
travel ``along'' $\Md V(\Bz)$ in order to make the distance to $\Bw\not\in V(\Bz)$
smaller. Indeed, even if we are travelling along one of the ``long
dimensions'' of the needle, a tiny generic perturbation of the direction of
travel will cause us to exit $V(\Bz)$ after an extremely short distance. For
example, when $n\approx 20$ our numerical experiments showed that we needed to
perform $10^{10}$ steps to make even a barely noticeable change in the
distance of $|\Bp-\Bw|$.

Graeme Milton \cite{gmpc20} suggested that since $V(\Bz)$ is a convex cone
which is dimensionally degenerate it must effectively lie in a low-dimensional
subspace of $\bb{C}^{n}$, in the same way as the needle whose point is at the
origin, effectively lies in a one-dimensional subspace of $\bb{R}^{3}$. In
order to capture this low-dimensional subspace (or rather its orthogonal
complement) we look for vectors $\BGx=(\xi_{1},\ldots,\xi_{n})\in\bb{C}^{n}$,
such that $|\BGx|=1$ and $\re(\Bw,\BGx)$ is negligibly small for all $\Bw\in
V(\Bz)$ with $|\Bw|=1$. Let us explore this idea. 

Suppose $\Gg\ge 0$ and $\Gs$
is the Stieltjes spectral measure. For given nodes $z_{j}\in\bb{H}_{+}$ we
define
\[
w_{j}[\Gs,\Gg]=\Gg+\int_{0}^{\infty} \frac{d\Gs(t)}{t-z_{j}}.
\]
We estimate
\begin{equation}
  \label{wtonorm}
|w_{j}[\Gs,\Gg]|\le\Gg+\|\Gs\|\max_{t\ge 0}\left|\frac{t+1}{t-z_{j}}\right|\le
M(z_{j})(\Gg+\|\Gs\|),
\end{equation}
where
\[
\|\Gs\|=\int_{0}^{\infty}\frac{d\Gs(t)}{t+1},\qquad M(z_{j})=\max_{t\ge 0}\left|\frac{t+1}{t-z_{j}}\right|.
\]
It is not hard to compute the constant $M(z)$ explicitly, when $\im(z)>0$, using the theory of
fractional-linear maps.
We can also derive the reverse estimate from the formulas
\[
\im(w_{j})=\im(z_{j})\int_{0}^{\infty}\frac{d\Gs(t)}{|t-z_{j}|^{2}},
\]
and
\[
\re(w_{j})=\Gg+\int_{0}^{\infty}\frac{t-\re(z_{j})}{|t-z_{j}|^{2}}d\Gs(t)=
\Gg+\int_{0}^{\infty}\left|\frac{t+1}{t-z_{j}}\right|^{2}\frac{d\Gs(t)}{t+1}
-\frac{1+\re(z_{j})}{\im(z_{j})}\im(w_{j}).
\]
Denoting
\[
m(z_{j})=\min_{t\ge 0}\left|\frac{t+1}{t-z_{j}}\right|=\min\left\{\nth{|z_{j}|},1\right\},
\]
we obtain
\begin{equation}
  \label{normtow}
m(z_{j})(\Gg+\|\Gs\|)\le\frac{\im(w_{j}z_{j})+\im(w_{j})}{\im(z_{j})}\le
\frac{|z_{j}+1|}{\im(z_{j})}|w_{j}|.
\end{equation}
Inequalities (\ref{wtonorm}) and (\ref{normtow}) imply that there exist constants $c(\Bz)$ and $C(\Bz)$, such that
\begin{equation}
  \label{equivnorm}
c(\Bz)\|\Bw[\Gs,\Gg]\|_{\infty}\le\Gg+\|\Gs\|\le C(\Bz)\|\Bw[\Gs,\Gg]\|_{\infty},
\end{equation}
where
\begin{equation}
  \label{constants}
  \|\Bw\|_{\infty}=\max_{1\le j\le n}|w_{j}|,\quad 
c(\Bz)=\min_{1\le j\le n}\nth{M(z_{j})},\quad
C(\Bz)=\min_{1\le j\le n}\left\{\frac{|z_{j}+1|}{m(z_{j})\im(z_{j})}\right\}.
\end{equation}
This means that $\Gg+\|\Gs\|$ and $\|\Bw\|_{\infty}$ are equivalent norms of
$f\in\mathfrak{S}$, given by (\ref{Stielrep}), provided $\Bw=f(\Bz)$. 

We recall that our goal is to understand how the convex set
$
V_{1}(\Bz)=V(\Bz)\cap B(\Bzr,1)
$
would look like geometrically as a subset of the $2n$-dimensional Euclidean
space $\bb{C}^{n}$. We claim that this set, which is technically of full real
dimension $2n$ is ``very flat''. To quantify just how flat it is we look for
unit vectors $\BGx\in \bb{C}^{n}$, such that $\re(\Bw,\BGx)$ is very small for
all $\Bw\in V_{1}(\Bz)$.  We compute
\[
\re(\Bw,\BGx)=\re\left(\Gg S+\int_{0}^{\infty}\sum_{k=1}^{n}
\frac{\xi_{k}}{t-\bra{z_{k}}}d\Gs(t)\right),\quad S=\sum_{k=1}^{n}\xi_{k}.
\]
Since it is the measure $d\Gs(t)/(1+t)$ that is finite it will be convenient
to rewrite the above formula as follows:
\[
\re(\Bw,\BGx)=\re\left(\Gg S+\int_{0}^{\infty}(\psi[\BGx](t)+S)\frac{d\Gs(t)}{t+1}\right),
\]
where
\[
\psi[\BGx](t)=\sum_{k=1}^{n}\frac{\xi_{k}(\bra{z_{k}}+1)}{t-\bra{z_{k}}}.
\]
Thus,
\[
|\re(\Bw,\BGx)|\le(\Gg+\|\Gs\|)|\re(S)|+\|\Gs\|\max_{t\ge 0}|\re(\psi[\BGx](t))|.
\]
Since $\Gth[\BGx](t)=\re(\psi[\BGx](t))$ is a complicated function of $t$
whose maximum is impossible to compute directly we observe that both
$\Gth[\BGx]$ and $\Gth'[\BGx]$ are in $L^{2}(0,+\infty)$ and use the
inequality
\[
\max_{t\ge 0}|\Gth(t)|^{2}\le\|\Gth\|_{1,2}^{2}=\|\Gth\|_{L^{2}(0,+\infty)}^{2}+\|\Gth'\|_{L^{2}(0,+\infty)}^{2},
\]
valid for all $\Gth\in W^{1,2}(0,+\infty)$. The inequality is sharp. It
becomes equality when $\Gth(t)=e^{-t}$.
Hence,
\begin{equation}
  \label{Vflatness}
  |\re(\Bw,\BGx)|^{2}\le 2(\Gg+\|\Gs\|)^{2}\left(\re(S)^{2}
+\|\Gth[\BGx]\|_{1,2}^{2}\right)\le 2C(\Bz)^{2}\|\Bw\|_{\infty}^{2}\BQ(\Bz)[\BGx],
\end{equation}
where
\[
\BQ(\Bz)[\BGx]=\re(S)^{2}+\|\Gth[\BGx]\|_{1,2}^{2}
\] 
is a positive definite real quadratic form in $\BGx$ and $C(\Bz)$ is given in
(\ref{constants}), in accordance with (\ref{normtow}). Let
$\Gl_{1}>\Gl_{2}>\ldots>\Gl_{2n}>0$ be the eigenvalues of $\BQ(\Bz)$. For each
$\Gd_{m}=C(\Bz)\sqrt{2\Gl_{m+1}}$ taken as the ``negligibility threshold'', we
can regard $m$ as the effective dimension of $V(\Bz)$, since the
$2n-m$-dimensional span $\CW_{m}$ of all eigenvectors of $\BQ(\Bz)$ corresponding to
eigenvalues $\Gl_{k}$, $k>m$ is effectively orthogonal to $V(\Bz)$. Indeed, 
for any $\BGx\in \CW_{m}$ and
any $\Bw\in V_{1}(\Bz)$ we have the inequality\footnote{Obviously, the estimate
  holds in a larger convex subset $V(\Bz)\cap B_{\infty}(\Bzr,1)$ of $V(\Bz)$,
  where $B_{\infty}$ denotes a ball in $\|\cdot\|_{\infty}$ norm of
  $\bb{C}^{n}$.} $|\re(\Bw,\BGx)|\le\Gd_{m}$. For the example
(\ref{sqrtexample0}) the quadratic form is of full rank, its 40 eigenvalues
decreasing from $\Gl_{1}\approx 3.37\cdot 10^{8}$ to $\Gl_{40}=4.73\cdot
10^{-5}$. If the number of data points increases to 40: 
$
z_{j}=ie^{0.01+0.5j},\ j=0,1,\ldots,39
$
Then numerical rank of the $80\times 80$ matrix $\BQ(z)$ is 56. It also
remains 56 for the $100\times 100$ matrix $\BQ(z)$, corresponding to
$
z_{j}=ie^{0.01+0.4j},\ j=0,1,\ldots,49.
$
These results show that the theoretical bound (\ref{Vflatness}) is fairly
conservative and overestimates the perceived dimension of $V(\Bz)$ quite a
bit. 

The quadratic form $\BQ(\Bz)$ is not hard to compute explicitly using the
residue formula
\begin{equation}
  \label{resform}
  \int_{0}^{\infty}R(x)dx=-\sum_{r=1}^{N}{\rm Res}[R(z)\ln(-z),z=p_{r}],
\end{equation}
where $R(z)$ is a rational function  with at least $1/|z|^{2}$ decay at infinity and
poles $p_{r}$ none of which lie on $[0,+\infty)$.
Even with the exact formula for $\BQ(\Bz)$, the accurate computation of its
eigenvalues requires many more digits of precision than the floating point
allows even for $n=20$. In our examples we have used the Advanpix
Multiprecision Computing Toolbox for MATLAB (www.advanpix.com) using 200
digits of precision.
% See ~/rsch/kramers-kronig/Stieltjes/MPStieltjes/Qofz/

% \textbf{Figure??? shows the two dimensional cross-section of $V(\Bz)$ for the
%   same example (\ref{sqrtexample}) after orthogonally projecting onto
%   $\CW_{m}^{\perp}$, when $\Gd_{m}=???$}\\
% \textbf{Argue that either $W_{\Ge}$ neglects too much of $V(\Bz)$ or the problems we
% describe persist}

\section{The least squares algorithm}
\setcounter{equation}{0} 
\label{sec:lsq} 
In this section we describe the algorithm that solves the least squares
problem (\ref{lsqV}), displays the graph of the Caprini function $C(t)$
certifying that the minimum in (\ref{lsqV}) has indeed been reached (see
Theorem~\ref{th:Caprini}), and exhibit the ``uncertainty band'' where the
least squares minimizer might belong for different realizations of the random
noise in the data.

The first step in the algorithm is to replace $V(\Bz)$ by a much simpler
object: the positive span of an \emph{ad-hoc basis} of $V(\Bz)$.
\begin{definition}
  \label{def:adhocbasis}
An ad-hoc basis of $V(\Bz)$ is a finite set of positive spectral measures
$\mathfrak{B}=\{\Gs_{1},\ldots,\Gs_{N}\}$, whereby $V(\Bz)$ is replaced by
\begin{equation}
  \label{VBz}
V_{\mathfrak{B}}(\Bz)=\left\{\Bw\in\bb{C}^{n}:w_{j}=x_{0}+
\sum_{\Ga=1}^{N}x_{\Ga}\phi_{\Ga}(z_{j}),\ j=1,\ldots,n,\ x_{\Ga}\ge 0,\
\Ga=0,\ldots,N\right\},
\end{equation}
where
\[
\phi_{\Ga}(z)=\int_{0}^{\infty}\frac{d\Gs_{\Ga}(t)}{t-z},\ \Ga=1,\ldots,N.
\]
\end{definition}
The adjective ``ad-hoc'' indicates that our choice of the basis $\mathfrak{B}$
is nothing more than an educated guess, and other choices could be at least as
effective as our choice. The choice that appears to work well consists of
\begin{itemize}
\item measures $\Gd_{\tau}(t)$---unit point mass at $t=\tau$, where $\tau$ is either
the real or the imaginary part of one of $z_{j}$ for some $j$,
\item measures $\chi_{[s_{1},s_{2}]}(t)dt$, where $s_{1}$ and $s_{2}$ is either one
of the $\tau$s or a mid-point between adjacent $\tau$s.
\end{itemize}
We will denote this construction of an ad-hoc basis by $\mathfrak{B}(\BGt)$,
where $\BGt$ stands for a list of $\tau$'s used in the above construction.

Imagining $V(\Bz)$ as a needle explains why the choice of an ad-hoc basis can
be fairly arbitrary. Indeed, selecting a point $\Bw_{0}$ at random inside a
needle and replacing the needle with the ray $\{s\Bw_{0}:s\ge 0\}$ gives a
fairly accurate representation of the needle. The more accurately we want to
approximate $V(\Bz)$ the more important the choice of an ad-hoc basis
becomes. Our choice above is just an attempt to tie the ad-hoc basis to the
data in a somewhat natural and algorithmic fashion. Many existing algorithms
(e.g., \cite{bouk15,wscc15}) make an effort of choosing a better basis, but in
the absence of any rigorous approximation error analysis they also remain
largely ad-hoc. In the new algorithm the ad-hoc basis is only needed as a
stepping stone for the construction of a much better basis tailor-made for the
specific experimental data.

Once the above ad-hoc basis has been chosen, we compute
\[
p_{j}(\Bx)=x_{0}+\sum_{\Ga=1}^{N}x_{\Ga}\phi_{\Ga}(z_{j}),\quad j=1,\ldots,n,
\]
by solving the nonnegative least squares problem
\begin{equation}
  \label{lsqnonneg}
  \min_{\Bx\ge 0}|\Bp(\Bx)-\Bw|^{2}.
\end{equation}
The above least squares problem is solved by a well-established and
widely implemented nonnegative least squares algorithm \cite{laha95}.
\begin{figure}[t]
  \centering
  \includegraphics[scale=0.3]{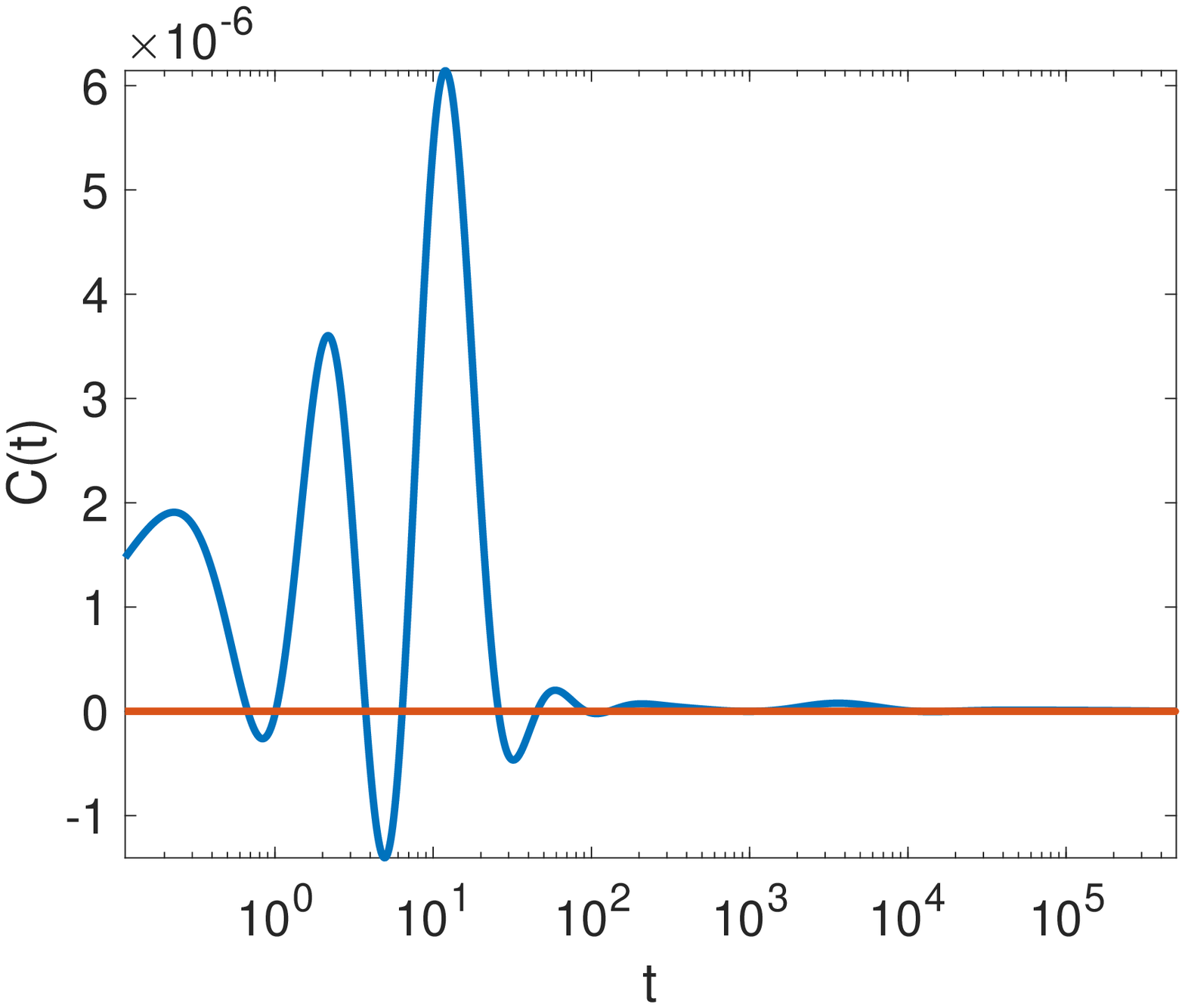}\hspace{10ex}
\includegraphics[scale=0.3]{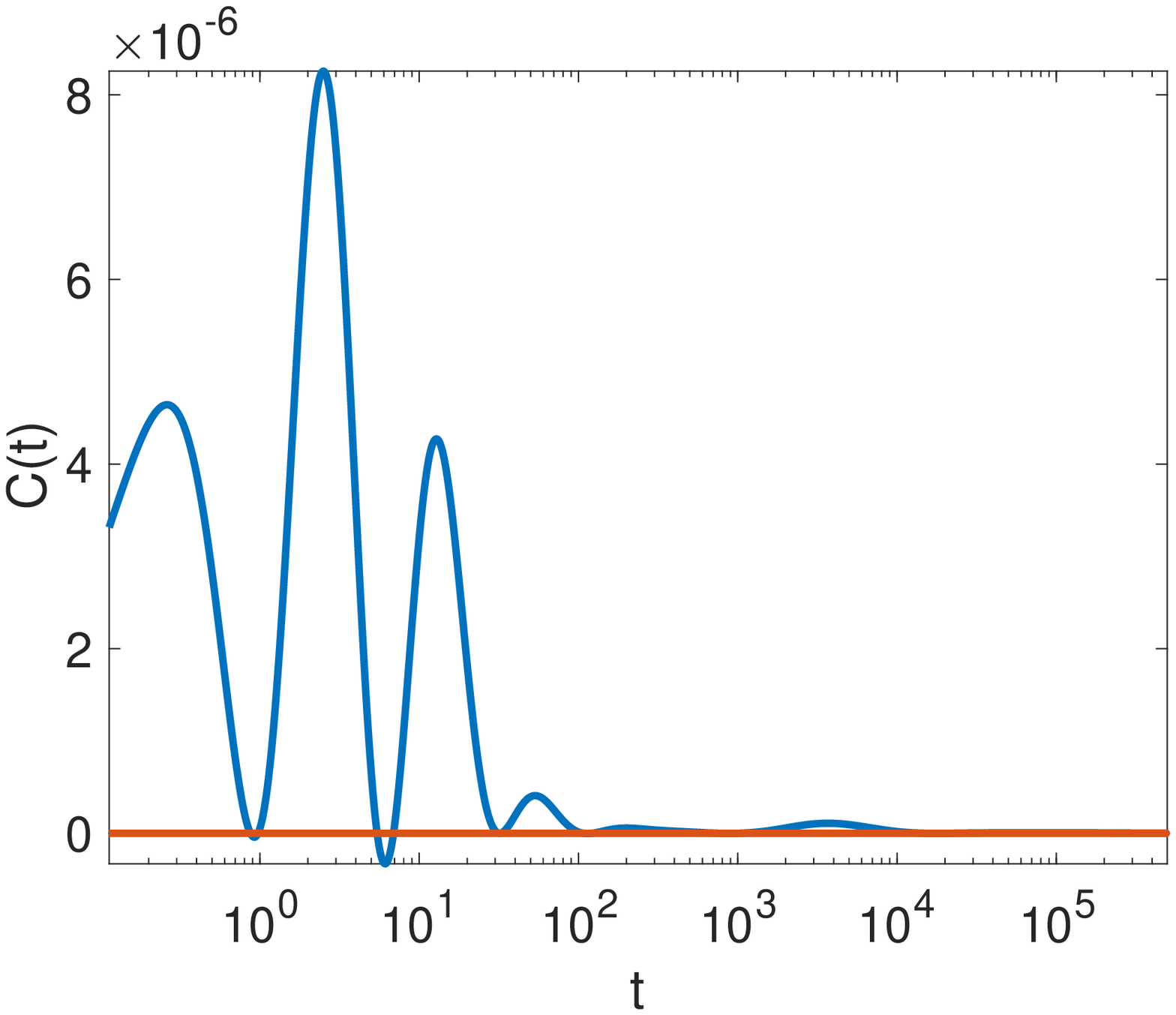}
% Made by rsch/matlab/kramers-kronig/Vofz_proj/Caprini0_plot.m
\caption{The Caprini function for the ad-hoc (left) and for the
  Caprini-augmented (right) bases projections.}
  \label{fig:Caprini0}
\end{figure}

Naturally, we would like to know how good our ad-hoc approximation is. For
illustration we once again turn to our simple example (\ref{sqrtexample0}). We
use the same noisy version $\Tld{\Bw}$ of $\Bw$ as in the example of
Figure~\ref{fig:Vcrossect}. The optimality conditions described in
Theorem~\ref{th:Caprini} require the Caprini function $C(t)$ to be nonnegative
and equal to zero on the support of the spectral measure. The graph of $C(t)$
shown in the left panel of Figure~\ref{fig:Caprini0} suggests that we are not
too far away from the true minimum, but are not there yet. Had we hit the
minimum exactly, the local minima of $C(t)$ would also be both the global
minima and the zeros of $C(t)$, and would comprise the support of the optimal
spectral measure $\Gs(t)$. This observation leads to the next step in our
algorithm: we add the points of local minima of $C(t)$ to the list of $\tau$'s
in our ad-hoc basis $\mathfrak{B}(\BGt)$ and recompute $\Bp(\Bx)$, solving
(\ref{lsqnonneg}) using the augmented ad-hoc basis $\mathfrak{B}(\BGt_{\rm
  aug})$ for $V(\Bz)$. The Caprini function for the new approximation is shown
in the right panel of Figure~\ref{fig:Caprini0}. We see both the substantial
improvement and the fact that the new approximation $\Bp^{*}$ is still not the
true minimum in (\ref{lsqV}). We can repeat this step by adjoining the local
minima of the improved Caprini function in the right panel of
Figure~\ref{fig:Caprini0} to the list of $\tau$'s. The improvement after the
second application of the augmentation of the ad hoc basis is significantly
smaller, and more repetitions no longer lead to discernible improvements.

To achieve certifiable optimality we cheat by ``moving the goalposts''. In the
author's experience the Caprini function is very sensitive to even the tiniest
deviations from the true optimum. The idea is to exploit this sensitivity and
achieve optimality by means of making negligible changes, but not in
$\Bp^{*}$, which is required to be in $V(\Bz)$. Changing $\Bw$ instead
of $\Bp^{*}$ leads to a \emph{linear problem}! We therefore look
for the \emph{alternative data} $\Tld{\Bw}$ near $\Bw$, so that the same $\Bp^{*}$ is a
true minimizer in (\ref{lsqV}), where $\Bw$ is replaced by $\Tld{\Bw}$, and
where $\Tld{\Bw}$ is computed by requiring that the local minima $t_{j}$ of
the original $C(t)$ satisfy equations (\ref{Copteqs}). In other words we are
looking for the vector $d\Bw\in\bb{C}^{n}$ of smallest norm, satisfying the
following equations:
\begin{equation}
  \label{Caprmod}
  \begin{cases}
\displaystyle
      \re\sum_{j=1}^{n}\frac{dw_{j}}{(t_{k}-\bra{z_{j}})^2}=0,\\
\displaystyle
      \re\sum_{j=1}^{n}\frac{dw_{j}}{t_{k}-\bra{z_{j}}}=C(t_{k}),\\
  \end{cases}\quad k=1,\ldots,N.
\end{equation}
If we want to enforce $\Gg>0$ condition we need to add the equation
\begin{equation}
  \label{gammapos}
  \re\sum_{j=1}^{n}dw_{j}=\re\sum_{j=1}^{n}(p_{j}-w_{j}).
\end{equation}
If $C(0)<0$ for the original data we add $t=0$ to the support of the spectral
measure $\Gs$ and require
\begin{equation}
  \label{zeronode}
  \re\sum_{j=1}^{n}\frac{dw_{j}}{\bra{z_{j}}}=-C(0).
\end{equation}
Vector $d\Bw$ can then be computed using the least norm least squares solver.

\begin{figure}[t]
  \centering
\includegraphics[scale=0.31]{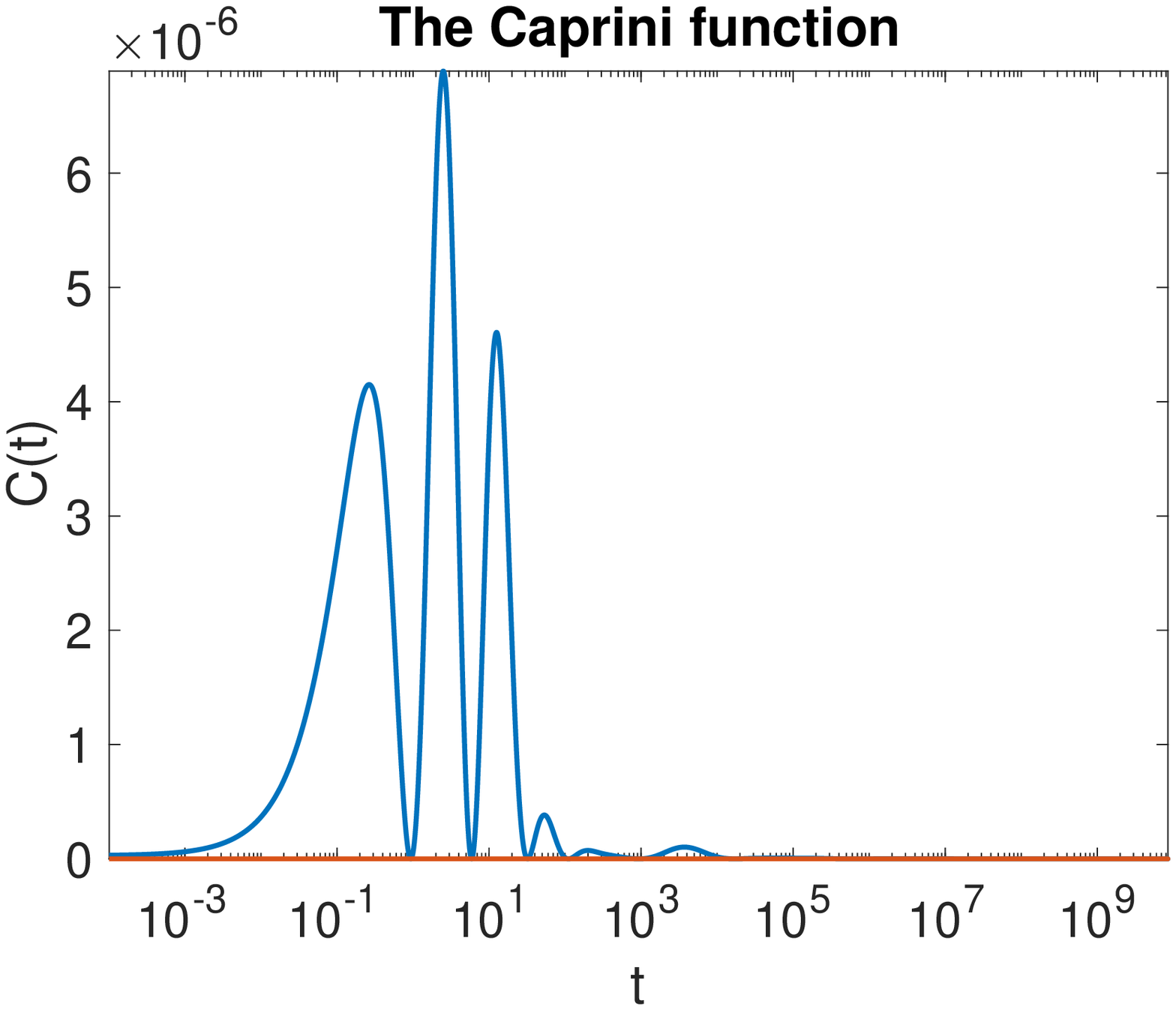}\hspace{10ex}
\includegraphics[scale=0.3]{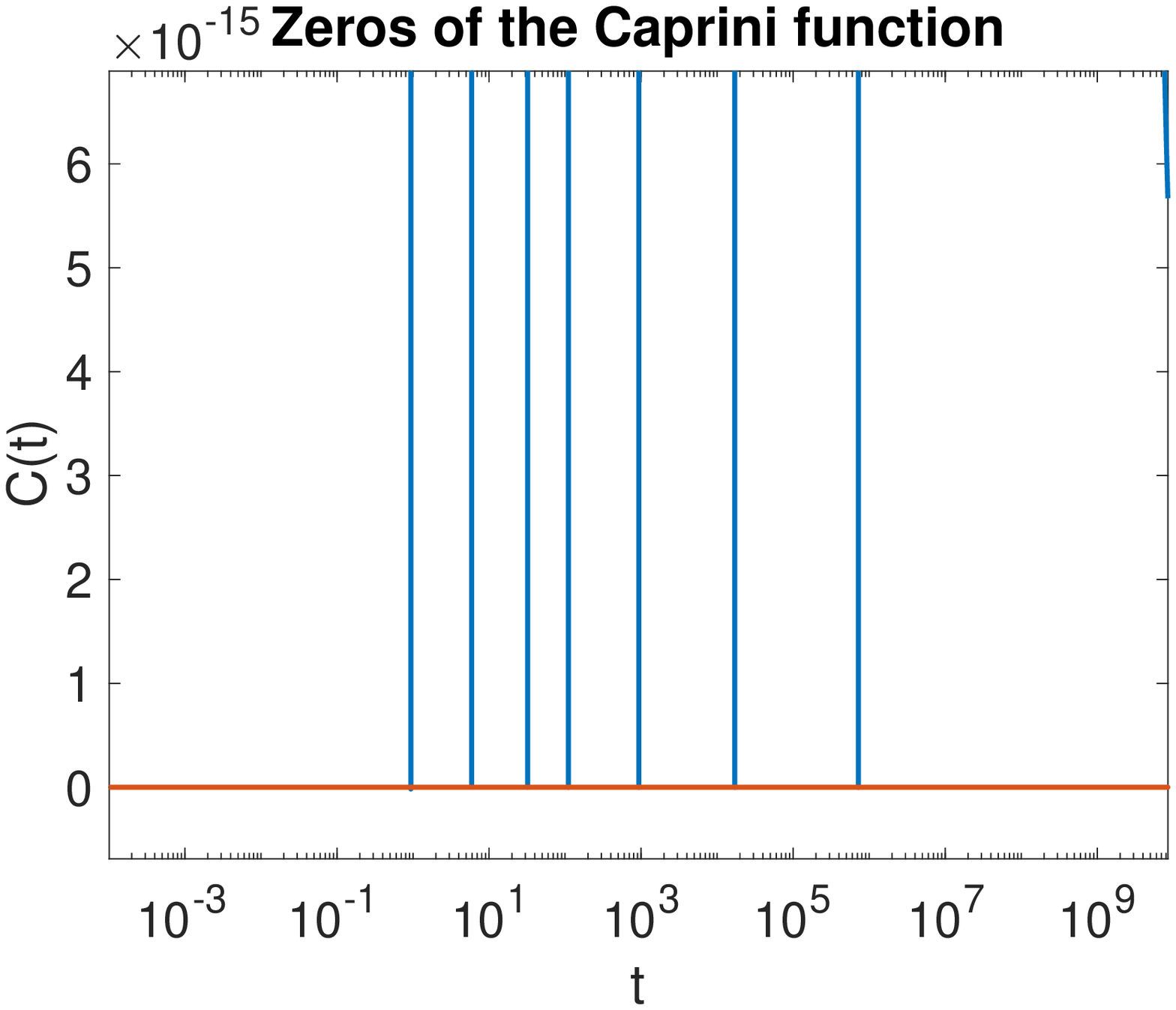}
% Made by rsch/kramers-kronig/Stieltjes/FORTRAN/Caprini0_plot.m and Voigt.f
  \caption{Achieving optimality for the ``alternative data''.}
  \label{fig:FullRun}
\end{figure}
Our simulations show that the ``alternative data'' $\Tld{\Bw}=\Bw+d\Bw$ is
indeed sufficiently close to the actual data to justify replacing one with the
other. In other words, if we regard $\Bw$ to be equal to $\Bp^{*}$ plus random
measurement errors, then $\Tld{\Bw}$ is also equal to $\Bp^{*}$ plus a
different realization of random measurement errors. At the same time the
Caprini function for the alternative data $\Tld{\Bw}$ in
Figure~\ref{fig:FullRun} shows that our formerly imperfect solution $\Bp^{*}$
of (\ref{lsqV}) is now optimal to within the computer precision\footnote{The
  right graph's vertical scale in Figure~\ref{fig:FullRun} is $10^{-9}$ times
  the right graph's vertical scale}, while $|\Bw-\Tld{\Bw}|/|\Bw|\approx
6.5\cdot10^{-4}$, where $\Bw$ is given by (\ref{sqrtexample0}) plus 2\% noise. 

On rare occasions during the algorithm testing the change from $\Bw$ to
$\Tld{\Bw}$ caused a point of local minimum $t=t_{j}$ of the original $C(t)$
to become a point of local maximum of the modified $C(t)$, while creating two
new points of local minima to the right and to the left of $t_{j}$. If the new
local minima are non-negligibly negative, then we update the list of local
minima of $C(t)$ and apply the same ``alternative data'' procedure to
$\Tld{\Bw}$, solving (\ref{Caprmod})--(\ref{zeronode}) again. In our numerical
tests no more than two iterations of ``data-fixing'' was ever necessary to
bring the graph of $C(t)$ into the desired shape.

In order to capture all local minima of $C(t)$ on $[0,+\infty)$ we
observe that $C(t)$ will be a monotone function on $[T,+\infty)$ for
sufficiently large $T$. Let us estimate the value of $T$. We will assume that
$\Gg>0$ and therefore
\[
\re\sum_{j=1}^{n}\Gd_{j}=0,\qquad\Gd_{j}=p_{j}-w_{j}.
\]
In this case we can write $C'(t)=D_{\infty}(t)+O(t^{-4})$, as $t\to\infty$, where
\[
D_{\infty}(t)=-\frac{2}{t^{3}}\re\sum_{j=1}^{n}\Gd_{j}\bra{z_{j}}.
\]
Estimating $|C'(t)-D_{\infty}(t)|$, it is not hard to show that
\begin{equation}
  \label{Cprest}
  |C'(t)-D_{\infty}(t)|<|D_{\infty}(t)|,\quad\forall t>T=(M_{0}+1)\max_{1\le j\le n}|z_{j}|,
\end{equation}
where
\[
M_{0}=\frac{2\sum_{j=1}^{n}|\Gd_{j}||z_{j}|}{|\re\sum_{j=1}^{n}\Gd_{j}\bra{z_{j}}|}.
\]
Inequality (\ref{Cprest}) shows that $C'(t)$ cannot be 0 when $t>T$. Hence, if
we want to make sure that we missed no local minima of $C(t)$ we need to
examine it only on the finite interval $[0,T]$.

In order to construct the function $f\in\mathfrak{S}$ satisfying
$f(\Bz)=\Bp^{*}$ we run the recursive interpolation algorithm described in
Section~\ref{sub:intalg}. In practice, even though matrices $\BN(\Bz,\Bp^{*})$ and
$\BP(\Bz,\Bp^{*})$ have no numerically significant negative eigenvalues,
feasibility gets lost after a number of iterations due to the amplification of
round-off errors. This may happen even when $n$ is as small as 10. When
this occurs, we replace the currently infeasible data $\Bw$ by its
``projection'' $\Bp^{*}$ as described above and continue the recursion using
the projected feasible data. 

Finally, our algorithm tries to estimate the degree of uncertainty of the
output. If we regard the discrepancies $w_{j}-f_{*}(z_{j})$ as a random noise,
then the fact that the measured values $w_{j}$ are exactly what they are is in part an
outcome of a random event. Simulating normal random noise with standard
deviation
\[
\rho^{2}=\nth{2n-1}\sum_{j=1}^{n}|w_{j}-f_{*}(z_{j})|^{2}
\]
we produce other ``realizations'' of the error of measurement, each of which
leads to its own least squares solution $f_{*}(z)$. Plotting these functions
for 500 different realizations of the random noise gives one an idea of the
degree to which we can trust the output of the algorithm. These potential
realizations are shown in grey in Figures~\ref{fig:Voigt} and
\ref{fig:Voigt_minus}. While in \cite{grho-annulus,grho-gen} we estimated the
\emph{worst case} error of extrapolation, these Monte-Carlo simulations are a
simple and direct way to estimate the uncertainty for \emph{specific
  data}. The use of Monte-Carlo simulations to exhibit the uncertainty in the
analytic continuation due to the statistical errors in the data has also been
used in particle physics \cite{acdi16}.

\section{Direct computation of spectral measure}
\setcounter{equation}{0} 
\label{sec:SM}
While the interpolation algorithm computes values $f(\Gz)$ for any specified
list of points $\Gz$ in the upper half-plane, one would also want to have an
explicit formula for $f(z)$. The goal of this section is to describe an
algorithm for computing the spectral representation (\ref{Frep}) of the
function $f\in\mathfrak{S}$ satisfying $f(\Bz)=\Bp$. The algorithm computes
this representation recursively following the algorithm described in
Section~\ref{sub:intalg}. It is based on the following theorem
\begin{theorem}
  \label{th:SM}
Suppose
\[
g(z)=\Gg_{g}-\frac{\Gs_{0}}{z}+\sum_{j=1}^{n}\frac{\Gs_{j}}{t_{j}-z},\qquad
\Gg_{g}\ge 0,\ \Gs_{0}\ge 0,\ \Gs_{j}>0,\
0< t_{1}<t_{2}<\dots<t_{n},
\]
Suppose $f(z)$ is given by (\ref{FviaM0}). Then
\[
f(z)=\Gg_{f}-\frac{\nu_{0}}{z}+\sum_{j=1}^{n+1}\frac{\nu_{j}}{\tau_{j}-z},
\]
where
\begin{equation}
  \label{gammanu}
  \Gg_{f}=\frac{\Gg_{*}\Gg_{g}}{\Gg_{g}+1},\qquad
\nu_{0}=\frac{\Gs_{0}\Gs^{*}}{\Gs_{0}+t_{*}},
\end{equation}
and
\[
0<\tau_{1}<t_{1}<\tau_{2}<t_{2}<\dots<t_{n}<\tau_{n+1}.
\]
\end{theorem}
\begin{proof}
  Formulas (\ref{gammanu}) are obtained by taking limits of $f(z)$ as
  $z\to\infty$ and $zf(z)$ as $z\to 0$ using formula (\ref{FviaM0}). We have
  also proved in Theorem~\ref{th:degred} that the degree of $f(z)$ is exactly
  1 higher than $g(z)$. Thus, proving that the intervals
  $(0,t_{1}),(t_{1},t_{2}),\ldots,(t_{n},+\infty)$ contain at least one pole
  of $f(z)$ would imply that these intervals must contain exactly one
  pole. Formula (\ref{FviaM0}) shows that the poles of $f(z)$ can only come
  either from the poles of $g(z)$ or from the zeros of the denominator
\[
\phi(z)=zg(z)+z-t_{*}.
\]
It is easy to compute that
\[
\lim_{z\to t_{j}}f(z)=\frac{\Gg_{*}t_{j}-\Gs^{*}}{t_{j}}\not=\infty.
\]
Hence, only the zeros of $\phi(z)$ can be the positive poles of $f(z)$. The
existence of zeros $\tau_{j}$ in the indicated intervals follows from the
following observations
\[
\lim_{x\to 0^{+}}\phi(x)=-\Gs_{0}-t_{*}<0,\quad
\lim_{x\to t_{j}^{\pm}}\phi(x)=\mp\infty,\quad\lim_{x\to+\infty}\phi(x)=+\infty.
\]
\end{proof}
Once the intervals containing single zeros of $\phi(x)$ are isolated, the zeros can be
computed using the standard zero finding algorithm \cite{brent73,fmm77}. We
only need to derive the upper bound for the last pole $\tau_{n+1}$. We observe
that all functions
\[
R_{j}(x)=\frac{x\Gs_{j}}{t_{j}-x},\quad j=1,\ldots,n
\]
are monotone increasing on $(t_{n},+\infty)$. Thus, when $x\ge 2t_{n}$ we have
\[
R_{j}(x)\ge-\frac{2t_{n}\Gs_{j}}{2t_{n}-t_{j}}\ge-2\Gs_{j}.
\]
Therefore,
\[
\phi(x)=(\Gg_{g}+1)x-t_{*}-\Gs_{0}+\sum_{j=1}^{n}R_{j}(x)\ge(\Gg_{g}+1)x-t_{*}-2\sum_{j=0}^{n}\Gs_{j}.
\]
We conclude that $\phi(x)>0$ when $x>T_{\max}$, where
\[
T_{\max}=\max\left\{2t_{n},(\Gg_{g}+1)^{-1}\left(t_{*}+2\sum_{j=0}^{n}\Gs_{j}\right)\right\}.
\]
The spectral representation of $f(z)$ is then computed recursively, using
(\ref{FviaM0}), with the
explicit formula in the case when
$
g(z)=\Gg_{g}-\Gs_{0}/z:
$
\begin{equation}
  \label{exitSM}
  f(z)=\frac{\Gg_{*}\Gg_{g}}{\Gg_{g}+1}-\frac{\Gs_{0}\Gs^{*}}{(\Gs_{0}+t_{*})z}
+\frac{\nu_{1}}{\tau_{1}-z},
\end{equation}
where
\[
\tau_{1}=\frac{\Gs_{0}+t_{*}}{\Gg_{g}+1},\quad
\nu_{1}=\frac{\Gs^{*}\Gg_{g}+\Gs_{*}+\Gg_{*}\Gs_{0}}{\Gg_{g}+1}
-\frac{\Gg_{*}\Gg_{g}(\Gs_{0}+t_{*})}{(\Gg_{g}+1)^{2}}-\frac{\Gs_{0}\Gs^{*}}{\Gs_{0}+t_{*}}.
\]
In our numerical simulations the values of $f_{*}(z)$ at specified points
computed from the spectral representation of $f_{*}(z)$ are indistinguishable
(graphically) from the values computed using the recursion algorithm from
Section~\ref{sub:intalg}.

\section{Case study: Electrochemical impedance spectroscopy}
\setcounter{equation}{0} 
\label{sec:Voigt}
Electrochemistry studies electrical behavior of systems where the motion of
charges occurs not only due to the applied electric field but also due to
chemical reactions that occur on sometimes vastly different time scales. One
of the key characteristics of such systems is the electrochemical impedance
spectrum $Z(\Go)$ that has the meaning of resistance to an applied sinusoidal
current. Combining the sine and cosine function into a complex exponential the
steady response of such system to the current $I(t)=e^{i\Go t}$ is the voltage
$U(t)=R(\Go)e^{i(\Go t+\phi(\Go))}$. The resistance $R(\Go)$ and the phase shift
$\phi(\Go)$ are combined into a single complex valued function
$Z(\Go)=R(\Go)e^{i\phi(\Go)}$---the electrochemical impedance spectrum
(EIS). The theory of electrochemical cells, including batteries, electrodes
and electrolytes \cite[2.1.2.3]{bamc05} says that $Z(\Go)$ has the spectral
representation
\begin{equation}
  \label{EIS}
  Z(\Go)=R_{\infty}+\int_{0}^{\infty}\frac{d\Gs(\tau)}{1+i\Go\tau},\qquad
\int_{0}^{\infty}\frac{d\Gs(\tau)}{1+\tau}<+\infty,
\end{equation}
where $\Gs$ is a positive Borel-regular measure on $[0,+\infty)$, called the
distribution of relaxation times (DRT). This formula shows that if $Z(\Go)$ is EIS,
then $Z(\Go)=f(-i\Go)$ for some $f\in\mathfrak{S}$. It is also a continuum
version of the complex impedance of an electrical circuit made of a series of
Voigt elements, each being a resistor and a capacitor connected in parallel.
\begin{definition}
  A Voigt circuit is an electrical circuit made of finitely many resistors and
  capacitors.
\end{definition}
The following theorem has long been known \cite{fost24,cauer26,cauer29} (see
also \cite[Statement~2, p.~196, Vol.~1]{cauer58}).
\begin{theorem}
  \label{th:VoigtZ}
The complex impedance functions $Z(\Go)$ of Voigt circuits are in one-to-one
correspondence with rational Stieltjes functions $f\in\mathfrak{S}_{\CR}$ via
$
Z(\Go)=f(-i\Go).
$
\end{theorem}

\begin{figure}[t]
  \centering
  \includegraphics[scale=0.24]{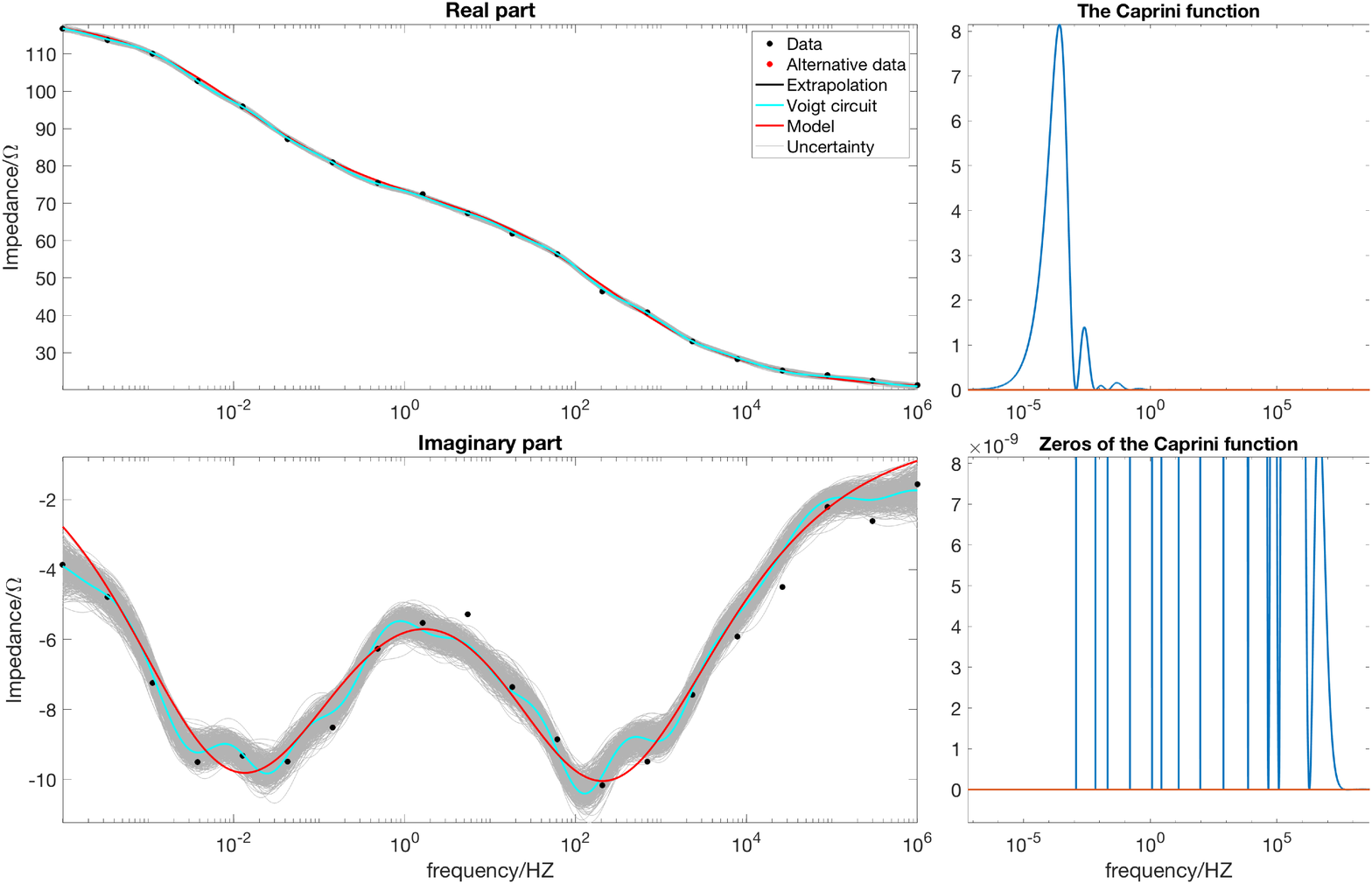}
% Made by rsch/kramers-kronig/Stieltjes/FORTRAN/Voigt4paper.m
% using data4paper.mat
  \caption{The output of the algorithm for a Voigt circuit.}
  \label{fig:Voigt}
\end{figure}
In electrochemistry there are several explicit EIS functions representing
important electrochemical cells, each serving as a building block of more
complex devices. The ideal capacitor's EIS $Z(\Go)=1/(iC\Go)$ is often
replaced by more realistic constant phase element (CPE) with
$Z_{\rm CPE}(\Go)=R/(i\tau\Go)^{\phi}$, $\phi\in[0,1]$.
Connecting it in parallel with a resistor gives the ZARC or Cole-Cole element
\[
Z_{\rm ZARC}(\Go)=\frac{R}{1+(i\tau\Go)^{\phi}},\quad\phi\in[0,1].
\]
A generalization of the ZARC element is the Havriliak-Negami element
\[
Z_{\rm HN}(\Go)=\frac{R}{(1+(i\tau\Go)^{\phi})^{\psi}},\quad\phi\in[0,1],\ \psi\in[0,1].
\]
Following examples in \cite{wscc15}, we test our algorithm on a double
Havriliak-Negami element
\begin{equation}
  \label{DHN}
  Z_{\rm DHN}(\Go)=R_{\infty}+\frac{R_{0}}{(1+(i\tau_{1}\Go)^{\phi})^{\psi}}
+\frac{R_{0}}{(1+(i\tau_{2}\Go)^{\phi})^{\psi}},
\end{equation}
where we chose
$
R_{\infty}=20,\ R_{0}=50,\ \phi=0.5,\ \psi=0.8,\ \tau_{1}=20,\ \tau_{2}=0.001.
$
This element operates on two very different times scales (20 seconds and 1 millisecond)
differing by four orders of magnitude.

The ``experimental data'' was produced by computing $Z_{\rm DHN}(2\pi f)$ at
20 frequencies $f_{j}$ equispaced on the logarithmic scale from
$f_{\min}=10^{-4}$Hz to $f_{\max}=10^{6}$Hz and then polluting the exact values
with 1\% random noise on the relative scale. Figure~\ref{fig:Voigt} shows the
result of the implementation of the algorithm. The real and imaginary parts of
the exact EIS function (\ref{DHN}) are shown in red. The imaginary part has
exactly two local minima at $1/(2\pi\tau_{1})$ and $1/(2\pi\tau_{2})$. Since
the random noise is complex-valued and $\im(Z_{\rm DHN})$ is 10 times
smaller than $\re(Z_{\rm DHN})$, the relative size of the noise for the
imaginary part is actually 10\%. This is why the algorithm's reconstructions
seems to be better for the real part than for the imaginary part. While
absolute errors of reconstruction for both the real and the imaginary parts
are the same, the relative errors differ by a factor of 10.

There is no discernible difference between the actual and the ``alternative
data'' for which the plots of the Caprini function at the global and local
scales show certified optimality. The grey band indicates the uncertainty
of the extrapolation shown by the cyan curve. The cyan curve is a plot of a
rational function whose spectral measure is supported on 20 points. It
coincides to a computer precision with values computed by the recursion
algorithm of Section~\ref{sub:intalg}.
\begin{figure}[t]
  \centering
\includegraphics[scale=0.2]{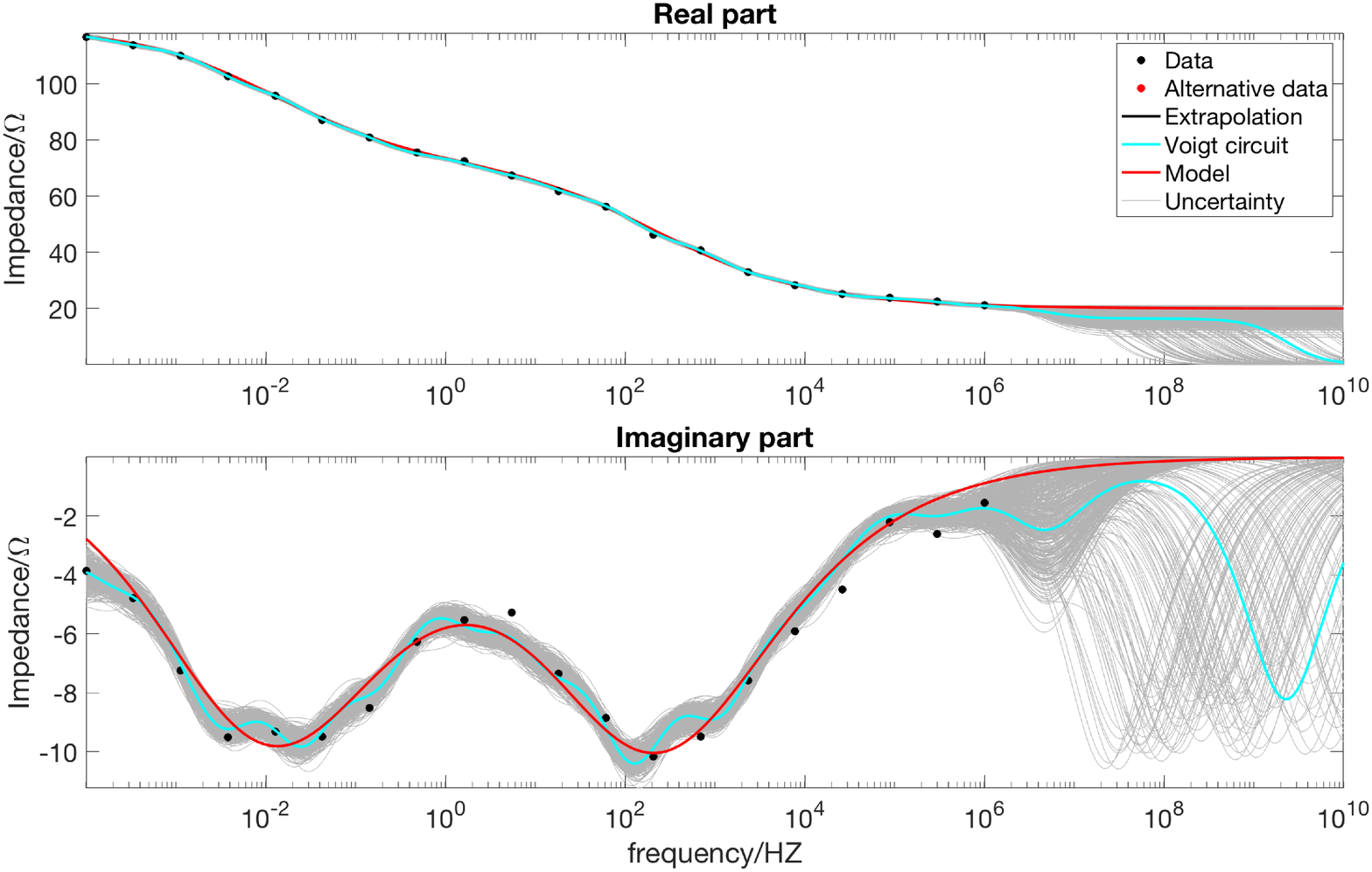}\ 
\includegraphics[scale=0.2]{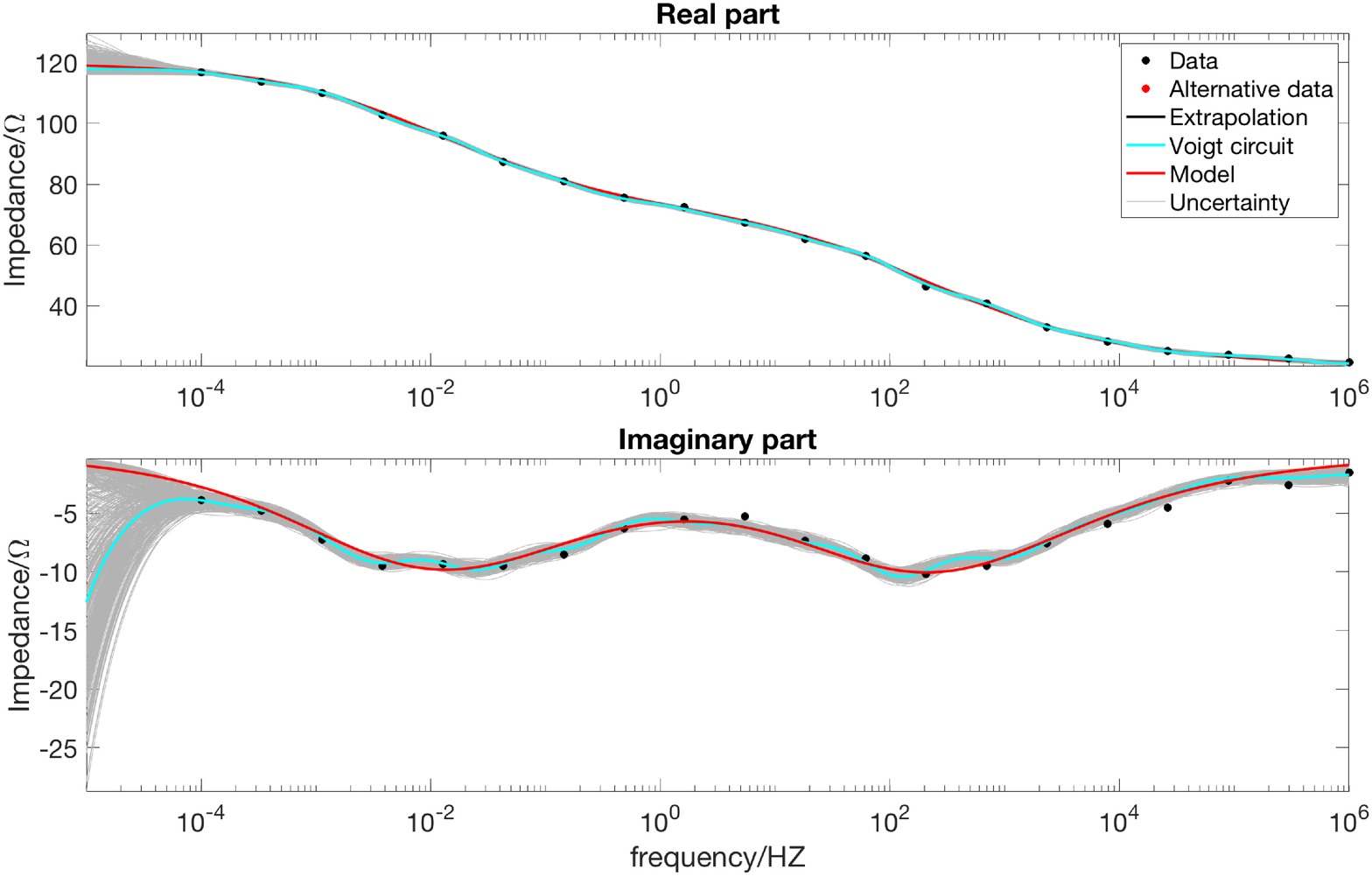}  
% Made by rsch/kramers-kronig/Stieltjes/FORTRAN/Voigt4paper.m
% uses data4paper.mat
  \caption{Extrapolation beyond the experimentally accessible frequency band.}
  \label{fig:Voigt_minus}
\end{figure}
It is important to keep in mind that the results in Figure~\ref{fig:Voigt}
look nice because we are ``filling the gaps'' between measurements. The
situation changes if we try to extrapolate beyond the largest or the smallest
frequency at which the impedance function has been
measured. Figure~\ref{fig:Voigt_minus} illustrates what happens with exactly
the same ``experimental data'' when we ask the algorithm to reconstruct the
EIS function on a larger frequency band. The uncertainty of reconstruction
``explodes'', but our two methods of extrapolation: the recursive and spectral
representation continue to agree. Both panels in Figure~\ref{fig:Voigt_minus}
show a pronounced disagreement between the theoretical and the extrapolated
curves away from the experimentally accessible frequency band, confirming that
it is in general impossible to extrapolate to the entire frequency spectrum reliably.

\medskip

\textbf{Acknowledgments.} The author is grateful to Graeme Milton, Mihai
Putinar, and Vladimir Bolotnikov for their comments and suggestions. A special
thanks goes to the referee who's detailed report have improved the paper
significantly. This material is based upon work supported by the National
Science Foundation under Grant No. DMS-2005538.

\def\cprime{$'$} \ifx \cedla \undefined \let \cedla = \c \fi\ifx \cyr
  \undefined \let \cyr = \relax \fi\ifx \cprime \undefined \def \cprime
  {$\mathsurround=0pt '$}\fi\ifx \prime \undefined \def \prime {'}
  \fi\def\Ya{Ya}


\begin{thebibliography}{10}

\bibitem{aogr95}
{\sc P.~Agarwal, M.~E. Orazem, and L.~H. Garcia-Rubio}, {\em Application of
  measurement models to impedance spectroscopy: Iii. evaluation of consistency
  with the {K}ramers-{K}ronig relations}, Journal of the Electrochemical
  Society, 142 (1995), p.~4159.

\bibitem{akhi21}
{\sc N.~I. Akhiezer}, {\em The classical moment problem and some related
  questions in analysis}, SIAM, 2021.

\bibitem{acd20}
{\sc B.~Ananthanarayan, I.~Caprini, and D.~Das}, {\em Test of analyticity and
  unitarity for the pion form-factor data around the {$\rho$} resonance},
  Physical Review D, 102 (2020), p.~096003.

\bibitem{acdi16}
{\sc B.~Ananthanarayan, I.~Caprini, D.~Das, and I.~S. Imsong}, {\em Precise
  determination of the low-energy hadronic contribution to the muon {$g- 2$}
  from analyticity and unitarity: An improved analysis}, Physical Review D, 93
  (2016), p.~116007.

\bibitem{bafa00}
{\sc A.~Bard and L.~Faulkner}, {\em Electrochemical Methods; Fundamentals and
  Applications}, Wiley Interscience Publications, 2000.

\bibitem{bamc05}
{\sc E.~Barsoukov and J.~R. Macdonald}, eds., {\em Impedance spectroscopy:
  theory, experiment, and applications}, John Wiley \& Sons Inc., 2nd~ed.,
  2005.

\bibitem{bdmh19}
{\sc D.~Batenkov, L.~Demanet, and H.~N. Mhaskar}, {\em Stable soft
  extrapolation of entire functions}, Inverse Problems, 35 (2019), p.~015011.

\bibitem{beto17}
{\sc B.~Beckermann and A.~Townsend}, {\em On the {S}ingular {V}alues of
  {M}atrices with {D}isplacement {S}tructure}, SIAM J. Matrix Anal. Appl., 38
  (2017), pp.~1227--1248.

\bibitem{berg78}
{\sc D.~J. Bergman}, {\em The dielectric constant of a composite material ---
  {A} problem in classical physics}, Phys. Rep., 43 (1978), pp.~377--407.

\bibitem{bosa99}
{\sc V.~Bolotnikov and L.~Sakhnovich}, {\em On an operator approach to
  interpolation problems for {S}tieltjes functions}, Integral Equations and
  Operator Theory, 35 (1999), pp.~423--470.

\bibitem{bouk95}
{\sc B.~A. Boukamp}, {\em A linear {K}ronig-{K}ramers transform test for
  immittance data validation}, Journal of the electrochemical society, 142
  (1995), p.~1885.

\bibitem{bouk15}
{\sc B.~A. Boukamp}, {\em Fourier transform distribution function of relaxation
  times; application and limitations}, Electrochimica acta, 154 (2015),
  pp.~35--46.

\bibitem{bouk20}
{\sc B.~A. Boukamp}, {\em Distribution (function) of relaxation times,
  successor to complex nonlinear least squares analysis of electrochemical
  impedance spectroscopy?}, Journal of Physics: Energy, 2 (2020), p.~042001.

\bibitem{brent73}
{\sc R.~P. Brent}, {\em Algorithms for Minimization Without Derivatives},
  Prentice-Hall, 1973.

\bibitem{brune31}
{\sc O.~Brune}, {\em Synthesis of a finite two-terminal network whose
  driving-point impedance is a prescribed function of frequency}, Journal of
  Mathematics and Physics, 10 (1931), pp.~191--236.

\bibitem{capr74}
{\sc I.~Caprini}, {\em On the best representation of scattering data by
  analytic functions in ${L}\sb{2}$-norm with positivity constraints}, Nuovo
  Cimento A (11), 21 (1974), pp.~236--248.

\bibitem{capr79}
{\sc I.~Caprini}, {\em Integral equations for the analytic extrapolation of
  scattering amplitudes with positivity constraints}, Nuovo Cimento A (11), 49
  (1979), pp.~307--325.

\bibitem{capr80}
{\sc I.~Caprini}, {\em General method of using positivity in analytic
  continuations}, Rev. Roumaine Phys., 25 (1980), pp.~731--740.

\bibitem{capr81}
{\sc I.~Caprini}, {\em Constraints on physical amplitudes derived from a
  modified analytic interpolation problem}, J. Phys. A, 14 (1981),
  pp.~1271--1279.

\bibitem{cauer26}
{\sc W.~Cauer}, {\em Die {V}erwirklichung von {W}echselstromwiderst{\"a}nden
  vorgeschriebener {F}requenzabh{\"a}ngigkeit}, Archiv f{\"u}r Elektrotechnik,
  17 (1926), pp.~355--388.

\bibitem{cauer29}
{\sc W.~Cauer}, {\em {\"U}ber eine {K}lasse von {F}unktionen, die die
  {S}tieltjesschen {K}ettenbr{\"u}che als {S}onderfall enth{\"a}lt.},
  Jahresbericht der Deutschen Mathematiker-Vereinigung, 38 (1929), pp.~63--72.

\bibitem{cauer58}
{\sc W.~Cauer}, {\em Synthesis of Linear Communication Networks}, vol.~I and
  II, 2nd Ed., McGraw-Hill, 1958.

\bibitem{chgo98}
{\sc E.~Cherkaeva and K.~M. Golden}, {\em Inverse bounds for microstructural
  parameters of composite media derived from complex permittivity
  measurements}, Waves Random Media, 8 (1998), pp.~437--450.

\bibitem{ciulli69}
{\sc S.~Ciulli}, {\em A stable and convergent extrapolation procedure for the
  scattering amplitude.---{I}}, Il Nuovo Cimento A (1965-1970), 61 (1969),
  pp.~787--816.

\bibitem{deto18}
{\sc L.~Demanet and A.~Townsend}, {\em Stable extrapolation of analytic
  functions}, Foundations of Computational Mathematics, 19 (2018),
  pp.~297--331.

\bibitem{digr01}
{\sc A.~Dienstfrey and L.~Greengard}, {\em Analytic continuation,
  singular-value expansions, and {K}ramers-{K}ronig analysis}, Inverse
  Problems, 17 (2001), p.~1307.

\bibitem{dyuka86}
{\sc Y.~M. Dyukarev and V.~Katsnelson}, {\em Multiplicative and additive
  classes of {S}tieltjes analytic matrix-valued functions and interpolation
  problems associated with them.}, Transactions of the American Mathematical
  Society, 131 (1986), pp.~55--70.

\bibitem{feyn64}
{\sc R.~P. Feynman, R.~B. Leighton, and M.~Sands}, {\em The {F}eynman lectures
  on physics. {V}ol. 2: {M}ainly electromagnetism and matter}, Addison-Wesley
  Publishing Co., Inc., Reading, Mass.-London, 1964.

\bibitem{fmm77}
{\sc G.~E. Forsythe, M.~A. Malcolm, and C.~B. Moler}, {\em Computer methods for
  mathematical computations.}, Prentice-Hall, Inc., Englewood Cliffs, N.J.,
  1977.

\bibitem{fost24}
{\sc R.~M. Foster}, {\em Theorems regarding the driving-point impedance of
  two-mesh circuits}, The Bell System Technical Journal, 3 (1924),
  pp.~651--685.

\bibitem{gold95a}
{\sc K.~M. Golden}, {\em Bounds on the complex permittivity of sea ice}, J.
  Geophys. Res. (Oceans), 100 (1995), pp.~699--711.

\bibitem{gopa83}
{\sc K.~M. Golden and G.~Papanicolaou}, {\em Bounds for effective parameters of
  heterogeneous media by analytic continuation}, Comm. Math. Phys., 90 (1983),
  pp.~473--491.

\bibitem{gopa85}
{\sc K.~M. Golden and G.~Papanicolaou}, {\em Bounds for effective parameters of
  multicomponent media by analytic continuation}, J. Statist. Phys., 40 (1985),
  pp.~655--667.

\bibitem{Stgh21}
{\sc Y.~Grabovsky}, {\em Fortran implementation of the {S}tieltjes function
  reconstruction algorithm}.
\newblock \url{https://github.com/YuryGrabovsky/Stieltjes}, February 2021.

\bibitem{grho-annulus}
{\sc Y.~Grabovsky and N.~Hovsepyan}, {\em Explicit power laws in analytic
  continuation problems via reproducing kernel {H}ilbert spaces}, Inverse
  Problems, 36 (2020), p.~035001.

\bibitem{grho-gen}
{\sc Y.~Grabovsky and N.~Hovsepyan}, {\em Optimal error estimates for analytic
  continuation in the upper half-plane}, Comm Pure Appl Math,  (2021).
\newblock to appear.

\bibitem{hmsw61}
{\sc J.~Hamilton, P.~Menotti, T.~Spearman, and W.~Woolcock}, {\em Evidence for
  pion-pion interactions froms-wave pion-nucleon scattering}, Il Nuovo Cimento
  (1955-1965), 20 (1961), pp.~519--528.

\bibitem{hojo85}
{\sc R.~A. Horn and C.~R. Johnson}, {\em Matrix Analysis}, Cambridge University
  Press, 1985.

\bibitem{kakr74}
{\sc I.~S. Kac and M.~G. Krein}, {\em R-functions--analytic functions mapping
  the upper halfplane into itself}, Amer. Math. Soc. Transl.(2), 103 (1974),
  p.~18.

\bibitem{kackr74}
{\sc I.~V. Kac and M.~G. Krein}, {\em On the spectral functions of the string},
  vol.~103 of Translations, Amer Mathematical Society, 1974.

\bibitem{khur57}
{\sc N.~N. Khuri}, {\em Analyticity of the {S}chr{\"o}dinger scattering
  amplitude and nonrelativistic dispersion relations}, Physical Review, 107
  (1957), p.~1148.

\bibitem{krnu98}
{\sc M.~Krein and A.~Nudelman}, {\em An interpolation problem in the class of
  {S}tieltjes functions and its connection with other problems}, Integral
  Equations and Operator Theory, 30 (1998), pp.~251--278.

\bibitem{Krein:1977:MMP}
{\sc M.~G. Krein and A.~A. Nudelman}, {\em The Markov Moment Problem and
  Extremal Problems}, Translation of Mathematical Monographs, 50, American
  Mathematical Society, Providence, RI, 1977.

\bibitem{lali60:8}
{\sc L.~D. Landau and E.~M. Lifshitz}, {\em Electrodynamics of continuous
  media}, vol.~8, Pergamon, New York, 1960.
\newblock Translated from the Russian by J. B. Sykes and J. S. Bell.

\bibitem{laha95}
{\sc C.~L. Lawson and R.~J. Hanson}, {\em Solving least squares problems},
  vol.~15 of Classics in Applied Mathematics, SIAM, 1995.

\bibitem{lipt01}
{\sc R.~Lipton}, {\em Optimal inequalities for gradients of solutions of
  elliptic equations occurring in two-phase heat conductors}, SIAM Journal on
  Mathematical Analysis, 32 (2001), pp.~1081--1093.

\bibitem{lspv05}
{\sc V.~Lucarini, J.~J. Saarinen, K.-E. Peiponen, and E.~M. Vartiainen}, {\em
  Kramers-{K}ronig relations in optical materials research}, vol.~110, Springer
  Science \& Business Media, 2005.

\bibitem{macd59}
{\sc S.~W. MacDowell}, {\em Analytic properties of partial amplitudes in
  meson-nucleon scattering}, Phys. Rev., 116 (1959), pp.~774--778.

\bibitem{mmdgbg96}
{\sc J.~V. Mantese, A.~L. Micheli, D.~F. Dungan, R.~G. Geyer,
  J.~Baker-{J}arvis, and J.~Grosvenor}, {\em Applicability of effective medium
  theory to ferroelectric/ferromagnetic composites with composition and
  frequency-dependent complex permittivities and permeabilities}, J. Appl.
  Phys., 79 (1996), pp.~1655--1660.

\bibitem{mmp21}
{\sc O.~Mattei, G.~W. Milton, and M.~Putinar}, {\em An extremal problem arising
  in the dynamics of two-phase materials that directly reveals information
  about the internal geometry}, Comm Pure Appl Math,  (2021).

\bibitem{mash16}
{\sc A.~Mecozzi, C.~Antonelli, and M.~Shtaif}, {\em Kramers-{K}ronig coherent
  receiver}, Optica, 3 (2016), pp.~1220--1227.

\bibitem{milt81b}
{\sc G.~W. Milton}, {\em Bounds on the complex permittivity of a two-component
  composite material}, J. Appl. Phys., 52 (1981), pp.~5286--5293.

\bibitem{mi81}
{\sc G.~W. Milton}, {\em Bounds on the transport and optical properties of a
  two-component composite material}, Journal of Applied Physics, 52 (1981),
  pp.~5294--5304.

\bibitem{milt17}
{\sc G.~W. Milton}, {\em Extending the Theory of Composites to Other Areas of
  Science}, Milton-Patton publishers, Salt Lake City, UT, USA, 2016.

\bibitem{gmpc20}
{\sc G.~W. Milton}, {\em Private communication}, 2020.

\bibitem{mem97}
{\sc G.~W. Milton, D.~J. Eyre, and J.~V. Mantese}, {\em Finite frequency range
  {K}ramers {K}ronig relations: bounds on the dispersion}, Phys. Rev. Lett., 79
  (1997), pp.~3062--3065.

\bibitem{nuss72}
{\sc H.~M. Nussenzveig}, {\em Causality and Dispersion Relations}, Academic
  Press, New York, 1972.

\bibitem{ocg12}
{\sc C.~Orum, E.~Cherkaev, and K.~M. Golden}, {\em Recovery of inclusion
  separations in strongly heterogeneous composites from effective property
  measurements}, Proc. R. Soc. Lond. Ser. A Math. Phys. Eng. Sci., 468 (2012),
  pp.~784--809.

\bibitem{ouch06}
{\sc M.-J. Ou and E.~Cherkaev}, {\em On the integral representation formula for
  a two-component elastic composite}, Math. Methods Appl. Sci., 29 (2006),
  pp.~655--664.

\bibitem{ou14}
{\sc M.-J.~Y. Ou}, {\em On reconstruction of dynamic permeability and
  tortuosity from data at distinct frequencies}, Inverse Problems, 30 (2014),
  p.~095002.

\bibitem{ssk93}
{\sc J.~Scully, D.~Silverman, and M.~Kendig}, eds., {\em Electrochemical
  Impedance: Analysis and Interpretation}, ASTM, 1993.

\bibitem{simon19}
{\sc B.~Simon}, {\em Loewner's Theorem on Monotone Matrix Functions}, Springer,
  2019.

\bibitem{sriv20}
{\sc A.~Srivastava}, {\em Causality and passivity: From electromagnetism and
  network theory to metamaterials}, Mechanics of Materials, 154 (2021),
  p.~103710.

\bibitem{trefe19}
{\sc L.~N. Trefethen}, {\em Quantifying the ill-conditioning of analytic
  continuation}, BIT Numerical Mathematics,  (2020).

\bibitem{wscc15}
{\sc T.~H. Wan, M.~Saccoccio, C.~Chen, and F.~Ciucci}, {\em Influence of the
  discretization methods on the distribution of relaxation times deconvolution:
  implementing radial basis functions with drttools}, Electrochimica Acta, 184
  (2015), pp.~483--499.

\bibitem{wobe65}
{\sc M.~Wohlers and E.~Beltrami}, {\em Distribution theory as the basis of
  generalized passive-network analysis}, IEEE Transactions on Circuit Theory,
  12 (1965), pp.~164--170.

\bibitem{zema72}
{\sc A.~H. Zemanian}, {\em Realizability Theory for Continuous Linear Systems},
  Academic Press, New York, NY, 1972.

\bibitem{zhch09}
{\sc D.~Zhang and E.~Cherkaev}, {\em Reconstruction of spectral function from
  effective permittivity of a composite material using rational function
  approximations}, J. Comput. Phys., 228 (2009), pp.~5390--5409.

\bibitem{zpsp20}
{\sc M.~{\v{Z}}ic, S.~Pereverzyev, V.~Suboti{\'c}, and S.~Pereverzyev}, {\em
  Adaptive multi-parameter regularization approach to construct the
  distribution function of relaxation times}, GEM-International Journal on
  Geomathematics, 11 (2020), p.~2.

\end{thebibliography}
\end{document}